\documentclass[a4paper]{amsart}
\usepackage[margin=3.3cm]{geometry}
\usepackage[utf8]{inputenc}
\usepackage[english]{babel}
\usepackage{amssymb}
\usepackage{amsmath}
\usepackage{amsthm}
\usepackage{amsfonts}
\usepackage{array}
\usepackage{comment,stmaryrd}
\usepackage{dirtytalk}
\usepackage{enumitem}
\usepackage{hyperref}
\usepackage{mathrsfs}
\usepackage{textcomp}
\usepackage{tikz}
\usetikzlibrary{arrows,decorations.markings,matrix,quotes}
\usepackage{tikz-cd, mathtools}
\usepackage{mathabx}
\usepackage{xfrac}
\usepackage[aligntableaux=center]{ytableau}
\usepackage{todonotes}
\usepackage{graphicx}
\graphicspath{ {./Images/} }

\addto\captionsenglish{}

\theoremstyle{plain}
\newtheorem{theorem}{Theorem}
\newtheorem{lemma}[theorem]{Lemma}

\newtheorem{proposition}[theorem]{Proposition}
\newtheorem{corollary}[theorem]{Corollary}

\theoremstyle{definition}
\newtheorem{definition}[theorem]{Definition}

\newtheorem{example}[theorem]{Example}

\theoremstyle{remark}
\newtheorem{remark}[theorem]{Remark}

\numberwithin{equation}{section}

\DeclareMathOperator{\N}{\mathbb{N}}
\DeclareMathOperator{\Q}{\mathbb{Q}}

\DeclareMathOperator{\C}{\mathbb{C}}
\DeclareMathOperator{\GL}{GL}

\DeclareMathOperator{\SL}{SL}
\DeclareMathOperator{\PSL}{\mathbb{P}SL}
\DeclareMathOperator{\Aff}{\mathbb{A}}
\DeclareMathOperator{\Prj}{\mathbb{P}}

\DeclareMathOperator{\im}{Im}

\DeclareMathOperator{\Hom}{Hom}
\DeclareMathOperator{\End}{End}
\DeclareMathOperator{\Ext}{Ext}

\DeclareMathOperator{\Mod}{Mod}
\DeclareMathOperator{\Coh}{Coh}

\DeclareMathOperator{\Per}{Per}
\DeclareMathOperator{\Db}{D^b}
\newcommand{\Dba}[1]{\mathrm{D}_{#1}^{\mathrm{b}}}
\newcommand*{\SHom}{\mathcal{H}\kern -.5pt om}
\newcommand*{\SEnd}{\mathcal{E}\kern -.5pt nd}
\newcommand*{\SExt}{\mathcal{E}\kern -.5pt xt}
\DeclareMathOperator{\Sym}{Sym}

\DeclareMathOperator{\Spec}{Spec}

\DeclareMathOperator{\rank}{rank}

\DeclareMathOperator{\Stab}{Stab}

\DeclareMathOperator{\Reg}{\mathcal{O}}

\DeclareMathOperator{\Pic}{Pic}

\DeclareMathOperator{\univ}{univ}
\DeclareMathOperator{\Rep}{Rep}
\DeclareMathOperator{\Hilb}{Hilb}
\newcommand{\GHilb}[1]{\mathop{\mbox{${#1}$-$\Hilb$}}}
\DeclareMathOperator{\Quot}{Quot}
\DeclareMathOperator{\NQV}{\mathfrak{M}}

\DeclareMathOperator{\Sing}{\mathbb{A}^2\!/\Gamma}
\newcommand{\Singa}[1]{\mathbb{A}^2\!/{#1}}
\DeclareMathOperator{\Ealg}{\mathcal{E}}
\DeclareMathOperator{\Sstack}{\mathscr{S}}

\DeclareMathOperator{\maxid}{\mathfrak{m}}

\DeclareMathOperator{\Lineb}{\mathcal{L}}
\DeclareMathOperator{\Fb}{\mathcal{F}}

\definecolor{burntorange}{rgb}{0.8, 0.33, 0.0}
\definecolor{burntumber}{rgb}{0.54, 0.2, 0.14}

\newcommand{\one}{\ensuremath{(\mathrm{i})}}
\newcommand{\two}{\ensuremath{(\mathrm{ii})}}
\newcommand{\three}{\ensuremath{(\mathrm{iii})}}

\newcommand{\CC}{\ensuremath{\mathbb{C}}} 
 
\newcommand{\NN}{\ensuremath{\mathbb{N}}} 
\newcommand{\QQ}{\ensuremath{\mathbb{Q}}}

\newcommand{\ZZ}{\ensuremath{\mathbb{Z}}} 
\newcommand{\M}{\ensuremath{\mathfrak{M}}} 

\newcommand{\Oo}{\ensuremath{\mathcal{O}}}

\newcommand{\git}{\ensuremath{\operatorname{\!/\!\!/\!}}}
\newcommand{\head}{\operatorname{h}}

\newcommand{\ltensor}{\overset{\mathbf{L}}{\otimes}}
 
\newcommand{\Mov}{\operatorname{Mov}}
\newcommand{\Nef}{\operatorname{Nef}}

\newcommand{\relint}{\operatorname{relint}} 

\newcommand{\Seshadri}{\operatorname{S}}

\newcommand{\tail}{\operatorname{t}}

\newcommand{\RHom}{\operatorname{\mathbf{R}Hom}}

\title{Stacky Resolutions of Kleinian Singularities and Nakajima Quiver Varieties}

\author{Lukas Bertsch}
\address{
Faculty of Mathematics, 
University of Vienna, 
Oskar-Morgenstern-Platz 1, 1090 Vienna, 
Austria}
\email{lukas.bertsch@univie.ac.at}

\author{Ruth Wye}
\address{Department of Mathematical Sciences,
University of Bath,
Claverton Down,
Bath,
BA2 7AY,
England}
\email{rep46@bath.ac.uk}

\begin{document}

\begin{abstract}
    We construct a class of noncommutative crepant resolutions of any Kleinian singularity as noncommutative sheaves of algebras over its crepant partial resolutions. We argue that such resolutions are Morita equivalent to the canonical orbifold resolutions of the partial resolutions. Further, we study Hilbert schemes of points on both the crepant partial and noncommutative resolutions and show that they are Nakajima quiver varieties. Finally, we describe the nef and movable cones of these Hilbert schemes of points using tautological bundles on the crepant partial resolutions.
\end{abstract}

\maketitle

\section*{Introduction}

Let $\Gamma < \SL(2,\C)$ be a finite subgroup, $\Sing$ be the associated Kleinian singularity, and \[f \colon S \to \Sing\] be its minimal resolution. The McKay correspondence \cite{mckay_graphs_80} establishes a relationship between the geometry of $S$ and the representation theory of $\Gamma$. One incarnation of this relationship \cite{kapranov_vasserot_kleinian-derived-hall_00} is an equivalence of derived categories \[\Db(S) \cong \Dba{\Gamma}(\Aff^2) \; .\] Here, the right-hand side is the bounded derived category of $\Gamma$-equivariant coherent sheaves on $\Aff^2$. Equivalently, this is the bounded derived category of coherent sheaves on the Deligne-Mumford stack $[\Sing]$, which, like $S$, is a crepant resolution of $\Sing$. Yet another crepant resolution is given in the sense of Van den Bergh \cite{van-den-bergh_nc-crepant_04} by the preprojective algebra $\Pi$, which is not just derived, but Morita equivalent to $[\Sing]$ \cite{reiten_van-den-bergh_two-dimensional-tame_89}: \[\Coh([\Sing]) \cong \Mod(\Pi) \; .\]

In this work, we study a larger family of crepant resolutions of $\Sing$ interpolating between $S$ and $[\Sing]$, respectively $\Pi$. We obtain these by first factoring $f$, as in the diagram below, through a partial resolution $S_K$, which we then resolve in a noncommutative or stacky way. \[\begin{tikzcd}[row sep = small] & S \arrow[dd,"f"] \arrow[ld,"h"'] \\ S_K \arrow[rd,"g"'] & \\ & \Sing \end{tikzcd}\] Here, the parameter is any subset $K \subseteq I\setminus\{0\}$, where $I$ is a set in bijection with the set of irreducible representations of $\Gamma$ and $0$ corresponds to the trivial representation.

After setting up some preliminary material, we introduce in section \ref{sec:nc-res} the sheaf of noncommutative algebras $h_* \SEnd(\mathcal{V})$ on $S_K$, where $\mathcal{V}$ is the tautological bundle on $S$ in the description of $S$ as the equivariant Hilbert scheme $\GHilb{\Gamma}(\Aff^2)$. This generalises the noncommutative resolution $\Pi$ since $\Pi = \End(\mathcal{V})$, to which $f_* \SEnd(\mathcal{V})$ is the associated sheaf on the affine variety $\Sing$. In section \ref{sec:stacky-res}, we introduce the DM stack $\Sstack_K$. Both $h_* \SEnd(\mathcal{V})$ and $\Sstack_K$ are crepant resolutions of $S_K$ and therefore of $\Sing$. We show the following equivalences.

\begin{theorem}[Theorem \ref{thm:derived-equivalence}, Corollary \ref{cor:nc-orbifold-equivalence}]\label{thm:intro:derived-equivalence}
    We have the following quasi-commutative diagram in which every functor is an equivalence of triangulated categories.
    \[\begin{tikzcd}
        \Db(S) \arrow[r] \arrow[d, "\mathcal{V} \otimes -"'] & \Db(\Sstack_K) \arrow[r] \arrow[d] & \Db([\Sing]) \arrow[d] \\ \Db(\SEnd(\mathcal{V})) \arrow[r, "Rh_*"] & \Db(h_*\SEnd(\mathcal{V})) \arrow[r, "Rg_*"] & \Db(\Pi)
    \end{tikzcd}\]
    Here, the vertical arrows preserve the natural t-structures, while the horizontal arrows modify them by a tilt.
\end{theorem}

The existence of equivalences in the top row follow from a theorem of Chen and Tseng \cite{chen_tseng_derived-mckay_08}, but here we establish them via the equivalences in the bottom row, which we prove following Van den Bergh \cite{van-den-bergh_flops-nc_04}. Given that, the least obvious part of Theorem \ref{thm:intro:derived-equivalence} is the Morita equivalence between $\Sstack_K$ and $h_* \SEnd(\mathcal{V})$, which we establish in Theorem \ref{thm:nc-orbifold-equivalence}.

In section \ref{sec:partial}, we also show the following analogue of Theorem \ref{thm:intro:derived-equivalence} for the partial resolutions.

\begin{theorem}[Theorem \ref{thm:partial}]\label{thm:intro:partial}
    We have the following equivalences of triangulated categories. \[\begin{tikzcd}
        \Db(S_K) \arrow[d, "\mathcal{T}_J \otimes -"'] & \\
        \Db(\SEnd(\mathcal{T}_J)) \arrow[r, "Rg_*"] & \Db(\Pi_J) \; .
    \end{tikzcd}\]
    Here, $\Pi_J$ is the preprojective algebra cornered in the subset $J \coloneqq I \setminus K$, $\mathcal{T}_J$ is the direct summand consisting of the $J$-components in $h_* \mathcal{V}$ and satisfies $\End(\mathcal{T}_J) \cong \Pi_J$, and $\SEnd(\mathcal{T}_J)$ is isomorphic to $h_*\SEnd(\mathcal{V})$ cornered in $J$. Again, the vertical arrow preserves the natural t-structure, while the horizontal arrow modifies it by a tilt.

    Furthermore, $S_K$ is a fine moduli space for $0$-generated $\Pi_J$-modules of dimension vector $\delta\vert_J$, and $\mathcal{T}_J$ is the universal family.
\end{theorem}

In section \ref{sec:hilbs}, we define the Hilbert schemes of points $\Hilb^n(\Sstack_K)$. In section \ref{sec:quots}, we define the schemes $\Quot_J^v$ parametrising certain zero-dimensional quotients of sheaves over $h_*\SEnd(\mathcal{V})$. We show that they are quasi-projective, smooth varieties and that the Morita equivalence between $\Sstack_K$ and $h_*\SEnd(\mathcal{V})$ gives rise to an isomorphism \[\Quot^{n \delta}_J \cong \Hilb^n(\Sstack_K) \; .\] In section \ref{sec:nqvs}, we introduce a class of Nakajima quiver varieties, which, originally due to Kronheimer \cite{kronheimer_ale-hk-quotients_89} and Nakajima \cite{nakajima_instantons_94}, are GIT quotients parametrising certain finite-dimensional framed $\Pi$-representations. We recall the VGIT wall-and-chamber structure described by Bellamy--Craw \cite{bellamy_craw_birational-symplectic_20}, and single out specific cones $C_K$ and $\sigma_K$. We prove the following.

\begin{theorem}[Theorem \ref{thm:isom-quot-nqv}]\label{thm:intro:isom-quot-nqv}
    For $\theta_K \in C_K$, and $\theta'_K$ in the relative interior of $\sigma_K$, we have the following commutative diagram in which the horizontal morphisms are isomorphisms and the vertical morphisms are resolutions of singularities. \[\begin{tikzcd} \Hilb^n(\Sstack_K) \arrow[r,"\sim"] \arrow[d] & \NQV_{\theta_K}(n \delta, \Lambda_0) \arrow[d] \\ \Hilb^n(S_K) \arrow[r, "\sim"] & \NQV_{\theta'_K}(n\delta, \Lambda_0) \end{tikzcd}\] Here, the horizontal maps are induced by the functors in Theorems \ref{thm:intro:derived-equivalence} and \ref{thm:intro:partial}, the morphism on the left is induced by pushforward of sheaves from $\Sstack_K$ to $S_K$, and the morphism on the right is a VGIT morphism.
\end{theorem}

In the statement above (and throughout the paper) we give $\Hilb^n(S_K)$ the underlying reduced structure instead of the natural scheme structure.\footnote{There is no known case where the natural scheme structure on $\Hilb^n(S_K)$ is non-reduced. In fact, it is known to be reduced for $n \leq 7$ \cite{craw_yamagishi_hilb-can_26}, and it seems reasonable to expect it is reduced in general, in which case the statement of Theorem \ref{thm:intro:isom-quot-nqv} would hold for the natural scheme structure. $\NQV_{\theta'_K}(n\delta, \Lambda_0)$, on the other hand, is known to be always reduced.} The bottom isomorphism was shown to exist by work of Craw with the second author \cite{craw_wye_hilb-crepant-partial_25-v2} using a combinatorial argument about the various GIT cones. Here we improve on the non-constructive proof and describe the morphism more explicitly. The isomorphism on top generalises the well-known isomorphisms $\GHilb{n\Gamma}(\Aff^2) \cong \NQV_{\theta_+}(n\delta,\Lambda_0)$ for $\theta_+ \in C_+$, and $\Hilb^n(S) \cong \NQV_{\theta_-}(n\delta,\Lambda_0)$ for $\theta_- \in C_-$ \cite{kuznetsov_quiver-hilb_07}. We deduce (Corollary \ref{cor:quot-from-nqv}) from existing wall-crossing results for Nakajima quiver varieties that the $\Hilb^n(\Sstack_K)$ are diffeomorphic and derived-equivalent for varying $K$. We also give the following application.

\begin{corollary}[Example \ref{ex:isom-eq-hilb-nqv}]\label{cor:normal-eq-hilb}
    Let $N \triangleleft \Gamma$ be a non-trivial, proper, normal subgroup. Let $Y$ be the minimal resolution of $\Singa{N}$, which naturally carries an action of $H \coloneqq \Gamma/N$. Then there exists a subset $K \subseteq I \setminus \{0\}$ such that, for $\theta_K \in C_K$, \[\GHilb{nH}(Y) \cong \NQV_{\theta_K}(n \delta, w) \; .\]
\end{corollary}

\noindent A special case of Corollary \ref{cor:normal-eq-hilb} is that of $N = \{\pm 1\}$, where we can identify $Y = T^{\vee}\!\Prj^1$. Here, $K$ is given explicitly in Example \ref{ex:morita-projective-mckay}.

Sections \ref{sec:geometric} and \ref{sec:proofsmooth} are dedicated to the proof of Theorem \ref{thm:intro:isom-quot-nqv}: In section \ref{sec:geometric}, we construct the isomorphism at the bottom by identifiying the coarse moduli space $\NQV_{\theta'_K}(n\delta, \Lambda_0)$ with a fine moduli space of framed modules over the cornered algebra $\Pi_J$. In section \ref{sec:proofsmooth}, we construct the isomorphism at the top and conclude the proof of Theorem \ref{thm:intro:isom-quot-nqv}. In the final section (section \ref{sec:cones}), we use the tautological bundles $\mathcal{T}_J$ to provide a description of the nef and moveable cones of $\Hilb^n(S_K)$ in terms of bundles on the surfaces $S_K$.

\subsection*{Acknowledgements}

The authors want to thank Alastair Craw, the second author's PhD advisor, for the insights he contributed to this project, as well as for his guidance and advice. The first author is grateful to his PhD advisors, Bal\'azs Szendr\H{o}i and \'Ad\'am Gyenge, for their support, and to them and S\o{}ren Gammelgaard for helpful comments on an earlier version of this paper. The second author would like to acknowledge the support of a PhD studentship awarded by the Heilbronn Institute for Mathematical Research, Grant number EP/V521917/1, from UKRI.

\numberwithin{theorem}{section}

\section{Noncommutative resolutions of Kleinian singularities}\label{sec:resolutions}

\subsection{Morita equivalence}

Throughout this paper, we will use various versions of Morita equivalence. We call two algebras Morita equivalent if their categories of finitely generated left modules are equivalent.

\begin{example}\label{ex:group-morita}
    An important instance of Morita equivalence is the following. Let $\Gamma$ be a finite group. Let $I$ be an index set in bijection $i \leftrightarrow \rho_i$ with the set of irreducible representations of $\Gamma$. Let $\delta \in \N^I$ be the vector of dimensions of the irreducible representations, that is, \[\delta_i \coloneqq \dim(\rho_i) \; .\] Then the group algebra $\C \Gamma$ is Morita equivalent to the semisimple commutative algebra $\C^I$, where the functor from $\Mod(\C^I)$ to $\Mod(\C \Gamma)$ is given by tensoring over $\C^I$ with the $(\C \Gamma,\C^I)$-bimodule $\bigoplus_{i \in I} \rho_i$, and the functor in the other direction is given by tensoring with the $(\C^I,\C \Gamma)$-bimodule $\bigoplus_{i \in I} \rho_i^{\vee}$. As a consequence, we have a natural isomorphism $\C \Gamma \cong \prod_{i \in I}\End_{\C}(\rho_i)$.
\end{example}

We extend the notion of Morita equivalence to sheaves of algebras on a given scheme: consider a coherent sheaf $\mathcal{A}$ on a Noetherian scheme $X$, equipped with the structure of an algebra with respect to the standard tensor product of quasi-coherent sheaves. Then there is a category $\Coh(\mathcal{A})$ of coherent sheaves on $X$ endowed with the structure of a left $\mathcal{A}$-module. We call two such sheaves of algebras $\mathcal{A}$ and $\mathcal{B}$ Morita equivalent if the categories $\Coh(\mathcal{A})$ and $\Coh(\mathcal{B})$ are equivalent.

The theorem below gives a general criterion for Morita equivalence of coherent sheaves of algebras in the presence of an $I$-grading. If $\mathcal{F} = \bigoplus_{i \in I} \mathcal{F}_i$ is any $I$-graded coherent sheaf on $X$, then $\mathcal{A} \coloneqq \SEnd_X(\mathcal{F})$ is a coherent sheaf of algebras on $X$ with idempotent global sections $e_i$ for $i \in I$ projecting $\mathcal{F}$ onto the summand $\mathcal{F}_i$. Accordingly, any sheaf $\mathcal{M}$ in $\Coh(\mathcal{A})$ decomposes as $\mathcal{M} = \bigoplus_{i \in I} e_i \mathcal{M}$. If its space of global sections is finite-dimensional, we call \[(\dim H^0(e_i\mathcal{M}))_{i \in I} \in \N^I\] the \emph{dimension vector} of $\mathcal{M}$. Analogously, we define the \emph{rank vector} of an $I$-graded locally-free sheaf of constant rank.

\begin{theorem}\label{thm:morita-general}
    Let $X$ be a Noetherian scheme, and let \[\mathcal{F} = \bigoplus_{i \in I} \mathcal{F}_i \quad \text{and} \quad \mathcal{G} = \bigoplus_{i \in I} \mathcal{G}_i\] be coherent sheaves graded by a fixed finite index set $I$. Let $a_i$ and $b_i$ be positive integers such that $(\mathcal{F}_i)^{a_i}$ is locally isomorphic to $(\mathcal{G}_i)^{b_i}$ for all $i \in I$. Define the coherent sheaves of $\C^I$-algebras \[\mathcal{A} \coloneqq \SEnd_X(\mathcal{F}) \quad \text{and} \quad \mathcal{B} \coloneqq \SEnd_X(\mathcal{G}) \; .\] Then there is a Morita equivalence \[\Coh(\mathcal{A}) \xleftrightharpoons[\; \Phi \;]{\; \Psi \;} \Coh(\mathcal{B}) \; ,\] where \[\Phi = \SHom_X(\mathcal{F}, \mathcal{G}) \otimes_{\mathcal{A}} - \quad \text{and} \quad \Psi = \SHom_X(\mathcal{G}, \mathcal{F}) \otimes_{\mathcal{B}} - \]
    
    Furthermore, if an object $\mathcal{M}$ of $\Coh(\mathcal{A})$ is supported on finitely many $\C$-points and has dimension vector $(l_i)_{i \in I} \in \N^I$, then the corresponding object of $\Coh(\mathcal{B})$ has dimension vector $(l_ia_i/b_i)_{i \in I}$. (In particular, $l_ia_i/b_i \in \N$ for all $i \in I$.) If $\mathcal{M}$ is locally-free of constant rank, its rank vector transforms in the same way.
\end{theorem}
\begin{proof}
    In order to show that $\Phi$ and $\Psi$ are quasi-inverse to each other and therefore define an equivalence, we show that, under the assumptions on the $\mathcal{F}_i$ and $\mathcal{G}_i$, the natural homomorphisms \[\SHom_X(\mathcal{F}, \mathcal{G}) \otimes_{\mathcal{A}} \SHom_X(\mathcal{G}, \mathcal{F}) \xrightarrow{\sim} \mathcal{B} \quad \text{and} \quad \SHom_X(\mathcal{G}, \mathcal{F}) \otimes_{\mathcal{B}} \SHom_X(\mathcal{F}, \mathcal{G}) \xrightarrow{\sim} \mathcal{A}\] defined by composition are isomorphisms. This can be checked locally as follows. Let $R$ be a commutative ring and $F = \bigoplus_i F_i$, $G = \bigoplus_i G_i$ be $I$-graded finitely generated $R$-modules such that $F_i^{a_i} \cong G_i^{b_i}$ for all $i \in I$. Let $A \coloneqq \End_R(F)$, and $B \coloneqq \End_R(G)$. Let $H \coloneqq \oplus_i F_i^{a_i} \cong \oplus_i G_i^{b_i}$, and $C \coloneqq \End_R(H)$. Then, any choice of projection from $H$ onto its summands $F$ and $G$ defines full idempotents $e_F, e_G \in C$ (\emph{full} meaning that $Ce_FC=C$ and $Ce_GC=C$) with $F \cong e_F H$ and $G \cong e_G H$. Then $A \cong e_F C e_F$, $B \cong e_G C e_G$, and \[\Hom_R(F,G) \otimes_A \Hom_R(G,F) \cong e_G C e_F \otimes_{e_F C e_F} e_F C e_G \cong e_GCe_G \cong B ,\] where the second isomorphism is proved in the same way as, e.g., in \cite[Proof of Cor.\ 2.3.4]{ginzburg_nc-geometry_05}. The same of course holds also with the roles of $F$ and $G$ reversed.

    To see how this equivalence transforms dimensions and ranks as stated, note that we have isomorphisms of locally defined right $e_F C e_F$-modules \[(e_{G_i}Ce_F)^{b_i} \cong e_iCe_F \cong (e_{F_i}Ce_F)^{a_i} \; .\] Hence, the $i$-th component of the image of a sheaf $\mathcal{M} \in \Coh(\mathcal{A})$ under the Morita equivalence is locally isomorphic to one of $b_i$ mutually isomorphic direct summands in the $a_i$-fold direct sum $(e_i\mathcal{M})^{a_i}$.
\end{proof}

\subsection{Kleinian singularities}

We work over $\C$. Let $\Gamma < \SL(2)$ be a finite subgroup. Let $\rho_{\text{given}}$ be its given two-dimensional representation. Let $u,v$ be the standard coordinates on $\Aff^2$. The quotient \[\Sing = \Spec(\C[u,v]^{\Gamma})\] has an isolated singularity at the image $o \in \Sing$ of the origin in $\Aff^2$ known as a \emph{Kleinian singularity}.

As in Example \ref{ex:group-morita}, let $I$ be an index set in bijection with the irreducible representations $\{\rho_i\}_{i \in I}$ of $\Gamma$. Furthermore, we fix a distinguished element $0 \in I$ corresponding to the trivial one-dimensional representation $\rho_0$. We denote the $i$-th idempotent in both $\C \Gamma$ and $\C^I$ by $e_i$.

The \emph{skew-group algebra} associated to the action of $\Gamma$ on $\Aff^2$, denoted by $\C[u,v] \rtimes \Gamma$, is constructed in such a manner that the category of left $\C[u,v] \rtimes \Gamma$-modules is equivalent to the category of $\Gamma$-equivariant quasicoherent sheaves on $\Aff^2$. The skew-group algebra contains $\C[u,v]$ and $\C \Gamma$ as subalgebras and is generated by them. Its center is $\C[u,v]^{\Gamma}$, the coordinate ring of $\Sing$.

The space $\C u \oplus \C v$ of linear functions on $\Aff^2$, as a representation of $\Gamma$, is dual (but isomorphic) to the given representation $\rho_{\text{given}}$. The \emph{McKay quiver} associated to this representation $\Gamma$ has vertex set $I$ and the number of arrows from $i \in I$ to $j \in I$ equals the multiplicity of $\rho_i$ in $\rho_{\text{given}}^{\vee} \otimes \rho_j$. The following was observed by McKay \cite{mckay_graphs_80}.

\begin{theorem}\label{thm:mckay-quiver}
    The McKay quiver associated to $\Gamma < \SL(2)$ is the doubled quiver of an extended ADE Dynkin diagram. This establishes a bijection between the finite subgroups of $\SL(2)$ up to conjugation and ADE Dynkin diagrams.
\end{theorem}

We now define the \emph{preprojective algebra} associated to $\Gamma$ (that is, the preprojective algebra associated to the extended Dynkin diagram). Take $E$ to be the set of arrows in the McKay quiver. For any $h \in E$, let $\Bar{h}$ denote the arrow in the opposite orientation. Now choose a subset $H \subset E$ of arrows such that $E = H \sqcup \Bar{H}$. The preprojective algebra $\Pi$ is the quotient of the path algebra of the McKay quiver by the two-sided ideal generated by the element \[\sum_{h \in H} [h,\Bar{h}] \; .\] The preprojective algebra inherits the structure of a graded $\C^I$-algebra. Furthermore, it is independent up to isomorphism of the choice of orientation. The following was shown in \cite[Propositions 2.4, 2.13]{reiten_van-den-bergh_two-dimensional-tame_89} (see also \cite[Theorem 0.1]{crawley-boevey_holland_nc-deformations-kleinian_98}).

\begin{theorem}\label{thm:morita-skew-preproj}
    The Morita equivalence between bimodules over $\C \Gamma$ and $\C^I$ induces Morita equivalence between $\C[u,v] \rtimes \Gamma$ and $\Pi$. Hence, we have equivalences \[\Coh_{\Gamma}(\Aff^2) \cong \Mod(\C[u,v] \rtimes \Gamma) \cong \Mod(\Pi) \; .\]
\end{theorem}

\subsection{The minimal resolution and the tautological vector bundle}

We have the following well-known description of the minimal resolution of $\Sing$ \cite[Theorem 1.3]{ito_nakamura_mckay-hilbert_96}.

\begin{theorem}
    Let $S \coloneqq \GHilb{\Gamma}(\Aff^2)$ be the Hilbert scheme of $\Gamma$-invariant zero-dimensional subschemes of $\Aff^2$ whose coordinate ring, as a $\Gamma$-representation, is isomorphic to the regular representation. There is a projective morphism \[ f \colon S \longrightarrow \Sing\] which maps a $\Gamma$-invariant subscheme of $\Aff^2$ to the image of its support under the quotient map $\pi \colon \Aff^2 \to \Sing$. The morphism $f$ is a minimal, crepant resolution of the singularity on $\Sing$.
\end{theorem}

The description of $S$ as a Hilbert scheme gives rise to a universal subscheme $\mathcal{Z} \subset S \times \Aff^2$ that fits into the following commutative diagram. \begin{equation}\label{eq:bkr-square}\begin{tikzcd}
    \mathcal{Z} \arrow[r, "q"] \arrow[d, "p"'] & \Aff^2 \arrow[d, "\pi"] \\ S \arrow[r, "f"] & \Sing 
\end{tikzcd}\end{equation}
Here, $\Gamma$ acts on $\mathcal{Z}$ in such a way that $q$ is $\Gamma$-equivariant, $p$ and $\pi$ are finite and $\Gamma$-invariant, and $p$ is flat. Furthermore, the variety $\Sing$ is normal, which implies \[f_* \Reg_S = \Reg_{\Sing} \; ,\] since the natural homomorphism from right to left corresponds to an integral extension of rings with the same fraction field.

Flatness and finiteness of $p$ implies that $p_* \Reg_{\mathcal{Z}}$ is a vector bundle on $S$. It is naturally equipped with a linear action by $\Gamma$ that is fibre-wise isomorphic to the regular representation. We obtain an isotypic decomposition \[p_*\Reg_{\mathcal{Z}} = \bigoplus_{i \in I} \rho_i \otimes \mathcal{V}_i \; , \quad \text{where} \quad \mathcal{V}_i \coloneqq \Hom_{\Gamma}(\rho_i, p_*\Reg_{\mathcal{Z}}) \; .\] Define $\mathcal{V} \coloneqq \bigoplus_{i \in I} \mathcal{V}_i$. We will refer to $\mathcal{V}$ and the $\mathcal{V}_i$ as \emph{tautological bundles}. Note that $p_* \Reg_{\mathcal{Z}}$ and $\mathcal{V}$ correspond to one another via the Morita equivalence between $\C \Gamma$ and $\C^I$.

\begin{proposition}[\cite{gonzales-sprinberg_verdier_construction_83}]\label{prop:tautb-basics}
    For each $i \in I$, $\mathcal{V}_i$ is a globally generated vector bundle of rank $\delta_i$. Furthermore, $\mathcal{V}_0 \cong \Reg_S$.
\end{proposition}
\begin{proof}
    The assertion $\rank(\mathcal{V}_i) = \delta_i$ follows from the fact that $\rho_i$ appears $\delta_i$ times in the regular representation of $\Gamma$. Second, since $\mathcal{Z}$ is a $\Gamma$-invariant subscheme of $S \times \Aff^2$ we have a $\Gamma$-equivariant surjection \[\psi \colon \Reg_S[u,v] \to p_* \Reg_{\mathcal{Z}}\] of sheaves of algebras on $S$. Since $\Reg_S[u,v]$ as a quasi-coherent sheaf is a direct sum of structure sheaves, this implies that $p_* \Reg_{\mathcal{Z}}$ is globally generated, and so are its direct summands $\mathcal{V}_i$. The section $\psi(1)$ is $\Gamma$-invariant and nowhere vanishing, so it spans $\mathcal{V}_0$.
\end{proof}

\begin{proposition}\label{prop:skew-preproj-as-end}
    We have canonical isomorphisms \[\End_S(p_* \Reg_{\mathcal{Z}}) \cong \C[u,v] \rtimes \Gamma \; , \quad \text{and} \quad \End_S(\mathcal{V}) \cong \Pi \; .\]
\end{proposition}
\begin{proof}
    The first isomorphism was defined and shown to be an isomorphism in \cite[Proposition 1.5]{kapranov_vasserot_kleinian-derived-hall_00}. It restricts to the identity on $\C \Gamma$, which is naturally contained as a subalgebra in both $\End_S(p_* \Reg_{\mathcal{Z}})$ and $\C[u,v] \rtimes \Gamma$. Hence, the first isomorphism reduces to the second via the Morita equivalence of Theorem \ref{thm:morita-skew-preproj}. 
\end{proof}

We define the sheaf of non-commutative algebras \[\Ealg \coloneqq \SEnd_S(\mathcal{V}) \; .\] By Proposition \ref{prop:skew-preproj-as-end}, we have an isomorphism $H^0(\Ealg) \cong \Pi$. Note that $\Ealg$ also contains the idempotents $e_i$ as global sections, and that we can identify $\mathcal{V} = \Ealg\!e_0$ and $\mathcal{V}^{\vee} = e_0\!\Ealg$.

We will also need the following important theorem \cite[Paragraph below Proposition 1.1]{auslander_purity_62}.

\begin{theorem}[Auslander]\label{thm:auslander}
    The natural homomorphism of $\C[u,v]^{\Gamma}$-algebras \[\C[u,v] \rtimes \Gamma \longrightarrow \End_{\C[u,v]^{\Gamma}}(\C[u,v])\] is an isomorphism.
\end{theorem}

\subsection{The complete-local setting}

Consider the completions \[\widehat{\Sing} \coloneqq \Spec(\C[[u,v]]^{\Gamma}) \quad \text{and} \quad \hat{S} \coloneqq S \times_{\Sing} \widehat{\Sing} \; ,\] and denote the base change of $f$ by \[\hat{f} \colon \hat{S} \to \widehat{\Sing} \; .\] Some aspects of the McKay correspondence have been developed in this complete-local setting, most notably Artin--Verdier theory (see \cite{artin_verdier_reflexive_85,van-den-bergh_flops-nc_04}, and Theorem \ref{thm:mckay}, items (\ref{thm:mckay:2}) and (\ref{thm:mckay:3}) below), which will play an important role in this paper.

Since $\widehat{\Sing}$ is flat over $\Sing$, and since flat base change commutes with blow-up \cite[\href{https://stacks.math.columbia.edu/tag/085S}{Tag 085S}]{stacks-project}, $\hat{f}$ can be obtained from $\widehat{\Sing}$ by iterated blow-up and is a minimal resolution. Furthermore, restriction (i.e., pullback) to the completions is subject to flat base change in cohomology \cite[Proposition III.9.3]{hartshorne_ag_77}. This says that for any coherent sheaf $\Fb$ on $S$, the natural homomorphism \[(R^k f_* \Fb)|_{\widehat{\Sing}} \longrightarrow R^k \hat{f}_*(\Fb|_{\hat{S}})\] is an isomorphism. In particular, $R^k f_* \Fb$ vanishes at $o$ if and only if $R^k \hat{f}_*(\Fb|_{\hat{S}})$ vanishes. We will use this result on multiple occasions to compare between the complete-local and the finite-type setting. We will also need the following elementary lemma.

\begin{lemma}\label{lem:sum-extension-formal}
    Let $Y$ be a Noetherian scheme, $y \in Y$ be a point, and $\hat{Y}$ be the completion of $Y$ at $y$. Let $\mathcal{V}$ and $\mathcal{W}$ be coherent sheaves on $Y$ such that there is an isomorphism \[\mathcal{V}|_{\hat{Y}} \cong \mathcal{W}|_{\hat{Y}} \; .\] Then there is an open neighbourhood $U$ of $y$ in $Y$ such that \[\mathcal{V}|_U \cong \mathcal{W}|_U\; .\]
\end{lemma}
\begin{proof}
    The completion $\hat{Y}$ is flat over $Y$, and flat pullback commutes with $\SHom$ \cite[Proof of Proposition 1.8]{hartshorne_reflexive_80}, so we have \[\SHom_{\hat{Y}}(\mathcal{V}|_{\hat{Y}}, \mathcal{W}|_{\hat{Y}}) = \SHom_Y(\mathcal{V},\mathcal{W})|_{\hat{Y}} \; .\] The given isomorphism is an element of this space, and on the right-hand side, over any affine open neighbourhood of $y$, we can find a section $\phi$ of $\SHom_Y(\mathcal{V},\mathcal{W})$ that is equal to our isomorphism modulo the maximal ideal at $y$. This means that $\phi$ is an isomorphism on the fibres at $y$, so it is surjective on an open neighbourhood of $y$ by Nakayama's lemma.

    Repeating the same argument with the roles of $\mathcal{V}$ and $\mathcal{W}$ reversed, we obtain a section $\psi$ of $\SHom_Y(\mathcal{W},\mathcal{V})$
    on an open neighbourhood of $y$ that is surjective and equals the given isomorphism on the fibre over $y$. We pick an affine open neighbourhood $U$ of $y$ on which $\phi$ and $\psi$ are defined and surjective. Then $\psi|_U \circ \phi|_U$ is a surjective endomorphism of $\mathcal{V}|_U$. From the well-known fact that a surjective endomorphism of a finitely-generated module is injective we deduce that $\psi|_U \circ \phi|_U$ is an isomorphism. In particular, $\phi|_U$ is injective and hence an isomorphism.
\end{proof}

\subsection{The McKay correspondence}

Let $\Db(S)$, $\Db(\Pi)$ and $\Dba{\Gamma}(\Aff^2)$ denote the bounded derived categories of coherent sheaves on $S$, finitely generated left $\Pi$-modules, and $\Gamma$-equivariant coherent sheaves on $\Aff^2$, respectively. The following statement brings together the results of \cite{mckay_graphs_80, artin_verdier_reflexive_85, kapranov_vasserot_kleinian-derived-hall_00, bkr_mckay-derived_01, van-den-bergh_flops-nc_04}.

\begin{theorem}[McKay correspondence]\label{thm:mckay}
    \
    \begin{enumerate}
        \item\label{thm:mckay:1} The underlying reduced subschemes of irreducible components of the exceptional fibre of $f$ are curves $C_i \cong \Prj^1$ in bijection with the set $I \setminus \{0\}$. For $i,j \in I \setminus \{0\}$, $C_i$ and $C_j$ intersect in at most one point, and they intersect if and only if $i$ and $j$ are adjacent in the McKay quiver. Furthermore, \[\deg(\mathcal{V}_i|_{C_j}) = \delta_{i,j} \; .\]
        
        \item\label{thm:mckay:2} The $\mathcal{V}_i|_{\hat{S}}$ are precisely the indecomposable globally generated vector bundles on $\hat{S}$ whose duals have vanishing $H^1$.
        
        \item\label{thm:mckay:3} The $f_* \mathcal{V}_i|_{\widehat{\Sing}}$, which are the sheaves associated to the modules $\Hom_{\Gamma}(\rho_i, \C[[u,v]])$ over $\C[[u,v]]^{\Gamma}$, are precisely the indecomposable reflexive coherent sheaves on $\widehat{\Sing}$.
        
        \item\label{thm:mckay:4} Following Proposition \ref{prop:skew-preproj-as-end}, we identify $\End(\mathcal{V}) = \Pi$. The functor \[\Phi \coloneqq \RHom(\mathcal{V}^{\vee},-) \colon \Db(S) \longrightarrow \Db(\Pi)\] is an equivalence of triangulated categories with quasi-inverse $\Psi \coloneqq \mathcal{V}^{\vee} \ltensor_{\Pi} -$.
    \end{enumerate}
\end{theorem}

Combined with the Morita equivalence from Proposition \ref{thm:morita-skew-preproj}, Theorem \ref{thm:mckay}(\ref{thm:mckay:4}) gives a derived equivalence \[\Db(S) \xlongrightarrow{\sim} \Dba{\Gamma}(\Aff^2) \; .\] This is isomorphic to the derived equivalence originally established in \cite{kapranov_vasserot_kleinian-derived-hall_00,bkr_mckay-derived_01}.

\begin{corollary}\label{cor:vanishing}
    For all $i,j \in I$ and $k \geq 1$ we have \[H^k(\SHom_S(\mathcal{V}_i, \mathcal{V}_j)) = 0 \; .\] In particular, $H^1(\Reg_S)=0$, $H^1(\Ealg) = 0$, $H^1(\mathcal{V}_i)=0$, and $H^1(\mathcal{V}^{\vee}_i)=0$ for all $i \in I$.
\end{corollary}
\begin{proof}
    Since $\Sing$ is affine, we have $H^k(\Fb) = H^0(R^kf_* \Fb)$ for any quasi-coherent sheaf $\Fb$ on $S$. This vanishes for $k \geq 2$ because $f$ has fibres of dimension $\leq 1$.
    
    It remains to show that $H^1(\SHom_S(\mathcal{V}_i, \mathcal{V}_j)) = 0$. Since $\mathcal{V}_j$ is globally generated, we have a short exact sequence of locally-free sheaves \[0 \to \mathcal{K} \to \Reg_S^N \to \mathcal{V}_j \to 0 \; .\] Applying $\SHom_S(\mathcal{V}_i,-)$ to this sequence, we obtain \[0 \to \SHom_S(\mathcal{V}_i,\mathcal{K}) \to (\mathcal{V}_i^{\vee})^N \to \SHom_S(\mathcal{V}_i, \mathcal{V}_j) \to 0 \; .\] 
    
    We have $H^1(\mathcal{V}_i^{\vee}) = H^0(R^1 f_* \mathcal{V}_i^{\vee}) = 0$ by Theorem \ref{thm:mckay}(\ref{thm:mckay:2}) and flat base change to the completion, so from the long exact sequence in cohomology and  $H^2(\SHom_S(\mathcal{V}_i,\mathcal{K}))=0$ we deduce that $H^1(\SHom_S(\mathcal{V}_i, \mathcal{V}_j)) = 0$.
\end{proof}

\begin{remark}
    In \cite{kapranov_vasserot_kleinian-derived-hall_00}, the authors argue that Corollary \ref{cor:vanishing} holds using merely the degrees of the tautological bundles along the exceptional curves. This is not possible, however, and there is a gap in their proof at the end of Remark 2.1. Indeed, Van den Bergh gives a counterexample in \cite[Theorem B]{van-den-bergh_flops-nc_04}, as he shows that in types D and E there exists a line bundle $\Lineb$ on $S$ such that there is some $i \in I$ with $\deg(\Lineb|_{C_j}) = \delta_{i,j}$ but $H^1(\Lineb^{\vee}) \neq 0$. (Van den Bergh shows this non-vanishing in the complete-local setting, but the same is true on the finite-type resolution by flat base change.) For this reason we circumvent the line of argument in \cite{kapranov_vasserot_kleinian-derived-hall_00} and instead deduce Corollary \ref{cor:vanishing} from Theorem \ref{thm:mckay}(\ref{thm:mckay:2}) which is part of the version of Artin--Verdier theory developed in \cite{van-den-bergh_flops-nc_04}.
\end{remark}

\noindent We will also later need the following Lemma, which is proved in the complete-local setting.

\begin{lemma}\label{lem:double-sequence}
    Let $i \in I$. There is a short exact sequence \[0 \to \mathcal{V}_i^{\vee}|_{\hat{S}} \to \Reg_{\hat{S}}^{2 \delta_i} \to \mathcal{V}_i|_{\hat{S}} \to 0 \; .\]
\end{lemma}
\begin{proof}
    For the rest of this proof we shall write $\mathcal{V}_i$ instead of $\mathcal{V}_i|_{\hat{S}}$ and the same for its dual. According to \cite[Lemma 3.5.1]{van-den-bergh_flops-nc_04}, taking $\delta_i - 1$ respectively $\delta_i + 1$ generic sections of $\mathcal{V}_i$ one obtains short exact sequences 
    \begin{equation}\label{eq:ses1}
       0 \to \Reg_{\hat{S}}^{\delta_i - 1} \to \mathcal{V}_i \to \Lineb_i \to 0
    \end{equation}
    and
    \begin{equation}\label{eq:ses2}
       0 \to \Lineb^{\vee}_i \to \Reg_{\hat{S}}^{\delta_i + 1} \to \mathcal{V}_i \to 0
    \end{equation}
    where $\Lineb_i = \det(\mathcal{V}_i)$ is a line bundle.

    Consider the following pushout, where $\iota$ appears in (\ref{eq:ses2}), and $\kappa$ appears in the dual of (\ref{eq:ses1}).
    \[\begin{tikzcd}
        \Lineb^{\vee}_i \arrow[r, "\iota"]\arrow[d , "\kappa"] & \Reg_{\hat{S}}^{\delta_i + 1} \arrow[d, "\lambda"] \\
        \mathcal{V}_i^{\vee} \arrow[r, "\mu"] & \mathcal{M}
    \end{tikzcd}\]
    Since $\iota$ and $\kappa$ are both inclusions of constant rank, their pushout $\mathcal{M}$ is also a vector bundle. Furthermore, the cokernels of $\iota$ and $\mu$ are isomorphic, and so are the cokernels of $\kappa$ and $\lambda$. Hence, we have short exact sequences
    \begin{equation}\label{eq:ses3}
       0 \to \mathcal{V}_i^{\vee} \xrightarrow{\mu} \mathcal{M} \to \mathcal{V}_i \to 0
    \end{equation}
    and 
    \begin{equation}\label{eq:ses4}
       0 \to \Reg_{\hat{S}}^{\delta_i + 1} \xrightarrow{\lambda} \mathcal{M} \to \Reg_{\hat{S}}^{\delta_i - 1} \to 0.
    \end{equation}
    Since $\Ext^1(\Reg_{\hat{S}}^{\delta_i - 1}, \Reg_{\hat{S}}^{\delta_i + 1})=H^1(\Reg_{\hat{S}}^{(\delta_i - 1)(\delta_i + 1)})=0$ by Corollary \ref{cor:vanishing}, the sequence \eqref{eq:ses4} splits, so $\mathcal{M} \cong \Reg_{\hat{S}}^{2\delta_i}$, and the desired sequence is given by (\ref{eq:ses3}).
\end{proof}

\subsection{A family of noncommutative crepant resolutions}\label{sec:nc-res}

Now, let $J \subseteq I$ be a subset containing $0$. We obtain a crepant partial resolution $S_K$ of $\Sing$ by contracting the exceptional curves in $S$ corresponding to elements of $K \coloneqq I \setminus J$. This fits into the following commutative diagram, where each morphism is projective and birational. \[\begin{tikzcd}[row sep = small] & S \arrow[dd,"f"] \arrow[ld,"h"'] \\ S_K \arrow[rd,"g"'] & \\ & \Sing \end{tikzcd}\] The singularities of $S_K$ are again of type ADE and $h$ is again a minimal resolution. In particular, the formal neighbourhood of a singular point in $S_K$ and its fibre are isomorphic to $\hat{f}$ for some different $\Gamma$. The surface $S_K$ is normal and we have $h_* \Reg_S = \Reg_{S_K}$.

Let $e_J \coloneqq \sum_{i \in J} e_i$ be the idempotent corresponding to the subset $J$. We write \[\mathcal{V}_J \coloneqq e_J \mathcal{V} = \bigoplus_{i \in J} \mathcal{V}_i \; , \quad \Ealg_J \coloneqq e_J\!\Ealg\!e_J = \SEnd_S(\mathcal{V}_J) \; , \quad \text{and} \quad \Pi_J \coloneqq e_J \Pi e_J = H^0(\Ealg_J) \; .\] These are referred to as the \emph{cornered} (sheaves of) modules and algebras with respect to $J$. We will need the following technical results.

\begin{lemma}\label{lem:push-taut-prop}
    \
    \begin{enumerate}
        \item\label{lem:push-taut-prop:1} $R^kh_* \Ealg = 0$ for $k \geq 1$.
        \item\label{lem:push-taut-prop:2} $h_* \mathcal{V}$ is a globally generated reflexive sheaf.
        \item\label{lem:push-taut-prop:3} $h_* \mathcal{V}_J$ is locally-free and $h^*h_* \mathcal{V}_J \cong \mathcal{V}_J$.
        \item\label{lem:push-taut-prop:4} The canonical homomorphism of sheaves of algebras $h_* \Ealg \to \SEnd_{S_K}(h_* \mathcal{V})$ is an isomorphism.
    \end{enumerate}
\end{lemma}
\begin{proof}
    \begin{enumerate}
        \item For $k \geq 2$ the statement follows from the fact that $h$ has fibres of dimension $\leq 1$. Therefore, it suffices to show the statement for $k=1$.
        
        By Corollary \ref{cor:vanishing}, we have $R^1 f_*(\mathcal{E}) = 0$. From the Grothendieck spectral sequence associated to the composition $Rf_* \cong Rg_* \circ Rh_*$ we deduce that $g_* R^1h_* (\mathcal{E}) = 0$. But since $R^1h_* (\mathcal{E})$ can only be supported on a finite number of points, this implies $R^1h_* (\mathcal{E}) = 0$.
        
        \item It suffices to show the statement for each $\mathcal{V}_i$, so let $i \in I$. For global generation, it is furthermore sufficient to show that the fibre of $h_* \mathcal{V}_i$ over any singular point of $S_K$ is generated by global sections, because we already know that $\mathcal{V}_i$ is globally generated and $h$ is an isomorphism locally over any smooth point of $S_K$.

        By Lemma \ref{lem:double-sequence}, we have a short exact sequence \[0 \to \mathcal{V}_i^{\vee}|_{\hat{S}} \to \Reg_{\hat{S}}^{2 \delta_i} \to \mathcal{V}_i|_{\hat{S}} \to 0\] of sheaves on $\hat{S}$. By item (\ref{lem:push-taut-prop:1}) and flat base change, respectively Theorem \ref{thm:mckay}(\ref{thm:mckay:2}), both $R^1h_*(\mathcal{V}_i^{\vee}|_{\hat{S}})$ and $R^1f_*(\mathcal{V}_i^{\vee}|_{\hat{S}})$ vanish. Hence, the pushforward of this sequence along $h$ and $f$, respectively, gives short exact sequences
        \begin{equation}\label{eq:push-double-ses-1}
            0 \to (h_* \mathcal{V}_i^{\vee})|_{\hat{S}_K} \to \Reg_{\hat{S}_K}^{2 \delta_i} \to (h_* \mathcal{V}_i)|_{\hat{S}_K} \to 0 \; ,
        \end{equation}
        on the formal neighbourhood $\hat{S}_K$ of the fibre of $o$ in $S_K$, and
        \begin{equation}\label{eq:push-double-ses-2}
            0 \to (f_* \mathcal{V}_i^{\vee})|_{\widehat{\Sing}} \to \Reg_{\widehat{\Sing}}^{2 \delta_i} \to (f_* \mathcal{V}_i)|_{\widehat{\Sing}} \to 0 \; .
        \end{equation}
        
        Since $\Sing$ is an affine variety, we can pick $2 \delta_i$ global sections $s_1, \ldots, s_{2 \delta_i} \in H^0(f_* \mathcal{V}_i)$ that agree with the sections in (\ref{eq:push-double-ses-2}) up to $\maxid_o H^0(f_* \mathcal{V}_i|_{\widehat{\Sing}})$, where $\maxid_o$ is the maximal ideal of the singular point. As $H^0(f_* \mathcal{V}_i) = H^0(h_* \mathcal{V}_i)$, we can interpret $s_1, \ldots, s_{2 \delta_i}$ as sections of $h_* \mathcal{V}_i$. As such, they agree with the sections in (\ref{eq:push-double-ses-1}) modulo $\maxid_o H^0(h_* \mathcal{V}_i|_{\hat{S}_K})$. But the elements of $\maxid_o H^0(h_* \mathcal{V}_i|_{\hat{S}_K})$ all vanish on $g^{-1}(o)$, so the $s_1, \ldots, s_{2 \delta_i}$ take the same value as the sections in (\ref{eq:push-double-ses-1}) at every point of $g^{-1}(o)$, and hence, also generate the fibre of $h_* \mathcal{V}_i$ at every point of $g^{-1}(o)$. In particular, they generate the fibre at the singular points of $S_K$, which are all contained in $g^{-1}(o)$. This concludes the proof that $h_* \mathcal{V}$ is globally generated.

        Reflexiveness of $h_* \mathcal{V}$ can be checked complete-locally at every closed point. Indeed, the natural homomorphism $h_* \mathcal{V} \to h_* \mathcal{V}^{\vee\!\vee}$ restricts to the corresponding natural homomorphism on the formal neighbourhood of a closed point $x \in S_K$, because taking duals commutes with restriction to the formal neighbourhood and whether a homomorphism of coherent sheaves is an isomorphism can be checked complete-locally around closed points as in the proof of Lemma \ref{lem:sum-extension-formal}.

        To see that $h_* \mathcal{V}$ is reflexive on formal neighbourhoods of closed points in $x$, note first that $h_* \mathcal{V}$ is locally-free and therefore reflexive on the smooth locus of $S_K$. Secondly, complete-locally around a singular point, $h$ is isomorphic to the minimal resolution of another Kleinian singularity, over the completion of which $\mathcal{V}$ decomposes into certain indecomposables according to Theorem \ref{thm:mckay}(\ref{thm:mckay:2}). The pushforward of these indecomposables are reflexive by Theorem \ref{thm:mckay}(\ref{thm:mckay:3}), and so is their direct sum.

        \item Let $i \in J$, and let $y \in S$ be a point in the exceptional locus of $h$. A choice of basis for the fibre of $\mathcal{V}_i$ at $y$ can be extended to global sections of $\mathcal{V}_i$. Since $\mathcal{V}_i$ is trivial along the exceptional locus of $h$ by Theorem \ref{thm:mckay}(\ref{thm:mckay:1}), these global sections trivialise $\mathcal{V}_i$ in an open neighbourhood of the exceptional locus. The first statement now follows from the fact that the natural homomorphism $h_* \Reg_S \to \Reg_{S_K}$ is an isomorphism, and the second statement follows from the fact that the natural homomorphism $h^* h_* \Reg_S \to \Reg_S$ is an isomorphism.

        \item First, note that $f_* p_* \Reg_{\mathcal{Z}} \cong \pi_* \Reg_{\Aff^2}$. Therefore, Theorem \ref{thm:auslander} combined with Proposition \ref{prop:skew-preproj-as-end} implies that the natural algebra homomorphism \[\End_S(p_* \Reg_{\mathcal{Z}}) \to \End_{\Sing}(f_* p_* \Reg_{\mathcal{Z}})\] is an isomorphism. Interpreting these algebras as coherent sheaves over $\Sing$ we see that the natural homomorphism \[f_* \SEnd_S(p_* \Reg_{\mathcal{Z}}) \longrightarrow \SEnd_{\Sing}(f_* p_* \Reg_{\mathcal{Z}})\] is an isomorphism. This isomorphism respects the decomposition of $p_* \Reg_{\mathcal{Z}}$ into copies of $\mathcal{V}_i$'s, so we deduce that the natural homomorphisms 
        \begin{equation}\label{eq:hom-commutes-with-push}
            f_* \SHom_S(\mathcal{V}_i, \mathcal{V}_j) \xlongrightarrow{\sim} \SHom_{\Sing}(f_* \mathcal{V}_i, f_* \mathcal{V}_j) \; .
        \end{equation}
        are isomorphisms.

        In order to show that the same statement holds for $h$ instead of $f$, we identify $h$ complete-locally around any singular point $x$ of $S_K$ with the minimal resolution $\hat{f_x} \colon \hat{S}_x \to \widehat{\Singa{\Gamma_x}}$ of some other type. It follows from item (\ref{lem:push-taut-prop:1}) and Theorem \ref{thm:mckay}(\ref{thm:mckay:2}) that the pullback of $\mathcal{V}_i$ to $\hat{S}_x$ again decomposes as a direct sum of tautological bundles. For these, we know that the natural homomorphism is an isomorphism using the completion of (\ref{eq:hom-commutes-with-push}) at $o$ with $\hat{f}_x$ instead of $\hat{f}$. This implies that (\ref{eq:hom-commutes-with-push}) with $h$ instead of $f$ is an isomorphism complete-locally around any point. Therefore, it is an isomorphism everywhere. \qedhere
    \end{enumerate}
\end{proof}

While the surface $S$ is a resolution of singularities of $\Sing$ in the classical sense, the sheaf of algebras $h_* \Ealg$ is a noncommutative crepant resolution in the sense of Van den Bergh \cite{van-den-bergh_nc-crepant_04,van-den-bergh_nc-crepant_22}. It follows from a theorem of Van den Bergh \cite{van-den-bergh_flops-nc_04} that all these resolutions are derived-equivalent:

\begin{theorem}\label{thm:derived-equivalence}
    The derived equivalence $\Phi$ from Theorem~\ref{thm:mckay} fits into the following quasi-commutative diagram in which every functor is an equivalence of triangulated categories. \[\begin{tikzcd} & \Db(\Ealg) \arrow[dd,"Rf_*"] \arrow[ld,"Rh_*"'] \arrow[r, equal, "\sim"] & \Db(S) \arrow[dd,"\Phi"] \\ \Db(h_* \Ealg) \arrow[rd,"Rg_*"'] & & \\ & \Db(f_* \Ealg) \arrow[r, equal, "\sim"] & \Db(\Pi) \end{tikzcd}\] where the two horizontal equivalences are obtained from equivalences of abelian categories.
\end{theorem}
\begin{proof}
    The commutativity of the diagram is clear. The vertical functor $\Phi$ is an equivalence by Theorem \ref{thm:mckay}(\ref{thm:mckay:4}). The equivalence $\Db(\Ealg) \cong \Db(S)$ is induced by the Morita equivalence between $\Ealg$ and $\Reg_S$ (which is a special case of Theorem \ref{thm:morita-general}). The equivalence $\Db(f_* \Ealg) \cong \Db(\Pi)$ follows from the equivalence $\Coh(f_* \Ealg) \cong \Mod(\Pi)$, which holds because $\Sing$ is affine and $H^0(f_* \Ealg) = \Pi$. It remains to show that $Rh_*$ is an equivalence of triangulated categories.
    
    It suffices to show that the bundle $\mathcal{V}$ satisfies the assumption of \cite[Proposition 3.3.1]{van-den-bergh_flops-nc_04} to be a local projective generator of $^{-1}\Per(S/S_K)$. By \cite[Proposition 3.2.7]{van-den-bergh_flops-nc_04}, this amounts to showing the following four properties:
    \begin{enumerate}
        \item\label{thm:derived-equivalence:pf:1} The natural homomorphism $h^* h_* \mathcal{V} \to \mathcal{V}$ is surjective.
        \item\label{thm:derived-equivalence:pf:2} $R^i h_*(\mathcal{V}^{\vee})=0$ for $i \geq 1$.
        \item\label{thm:derived-equivalence:pf:3} $\det(\mathcal{V})$ is ample.
        \item\label{thm:derived-equivalence:pf:4} $\Reg_X$ is a direct summand of $\mathcal{V}$.
    \end{enumerate}

    Property (\ref{thm:derived-equivalence:pf:1}) follows from the fact that $\mathcal{V}$ is globally generated (see Proposition \ref{prop:tautb-basics}). Property (\ref{thm:derived-equivalence:pf:2}) follows from Lemma \ref{lem:push-taut-prop}(\ref{lem:push-taut-prop:1}) because $\mathcal{V}^{\vee}$ is a direct summand of $\Ealg$. Property (\ref{thm:derived-equivalence:pf:3}) holds because $\det(\mathcal{V})$ has positive degree on each exceptional curve by Theorem \ref{thm:mckay}(\ref{thm:mckay:1}). Finally, property (\ref{thm:derived-equivalence:pf:4}) holds because $\mathcal{V}_0 \cong \Reg_S$.
\end{proof}

Note that we consider the functor from $\Db(S)$ to $\Db(h_* \Ealg)$ that, in Van den Bergh's theory, is the one associated with $\mathcal{V}^{\vee}$, the local projective generator for $^0\Per(S/S_K)$.

\subsection{The partial McKay correspondence}\label{sec:partial}

We now present the analogue of the McKay correspondence for any crepant partial resolution $S_K$ of a Kleinian singularity; we call this the \emph{partial McKay correspondence}. Set \[\mathcal{T}_J \coloneqq h_*\mathcal{V}_J \; .\] By Lemma \ref{lem:push-taut-prop}, $\mathcal{T}_J$ is a globally generated vector bundle on $S_K$, and we have \[\End_{S_K}(\mathcal{T}_J) = H^0(\SEnd_{S_K}(h_* \mathcal{V}_J)) \cong H^0(h_* \SEnd_S(\mathcal{V}_J)) = \Pi_J \; .\] This algebra acts fibre-wise on $\mathcal{T}_J$, and since $\mathcal{T}_J$ is globally generated, the natural homomorphism \begin{equation}\label{eq:universal-partial}
    e_J \Pi e_0 \otimes \Reg_{S_K} \longrightarrow \mathcal{T}_J
\end{equation}
is a surjection. In other words, every fibre of $\mathcal{T}_J$ is a $\Pi_J$-module of dimension vector $\delta|_J$ generated by its component at the vertex $0$. We now show that, just like $\mathcal{V}$, $\mathcal{T}_J$ is a tilting bundle and a universal family of $0$-generated $\Pi_J$-modules.

\begin{theorem}[The partial McKay correspondence]\label{thm:partial}
    \
    \begin{enumerate}
        \item\label{thm:partial:1} The functor \[\Phi_J \coloneqq \RHom(\mathcal{T}_J^{\vee}, -) \colon \Db(S_K) \longrightarrow \Db(\Pi_J)\] is an equivalence of triangulated categories with quasi-inverse $\Psi_J \coloneqq \mathcal{T}_J^{\vee} \ltensor_{\Pi_J} -$.
        
        \item\label{thm:partial:2} $S_K$ is a fine moduli space for $0$-generated $\Pi_J$-modules of dimension vector $\delta\vert_J$. \footnote{For a precise definition of the moduli functor, see \cite[section 2.6]{karmazyn_git-tilting_17}. In order to get rid of the $\C^*$-stabiliser of any module, Karmazyn identifies families that differ by the twist of a line bundle. An alternative but isomorphic functor would be the Quot functor parametrising quotients of $e_J \Pi e_0$ over $\Pi_J$.} 
    \end{enumerate}
\end{theorem}
\begin{proof}
    Part (\ref{thm:partial:1}) follows, exactly as in the proof of Theorem \ref{thm:derived-equivalence}, from \cite{van-den-bergh_flops-nc_04} and the properties of $\mathcal{V}_J$ given above. Part (\ref{thm:partial:2}) is a special case of Karmazyn's result \cite[Corollary 5.2.4]{karmazyn_git-tilting_17}.
\end{proof}

\section{Kleinian orbifolds}\label{sec:stacky-res}

\subsection{Stacky resolutions}

Another notion of noncommutative resolution lives in the world of stacks. For example, the quotient stack $[\Sing]$ resolves the singularity on $\Sing$ and, by Proposition \ref{thm:morita-skew-preproj}, $[\Sing]$ is Morita equivalent to the noncommutative algebra $\Pi$. In this section, we will generalise this Morita equivalence to noncommutative resolutions of $S_K$.

Since $S_K$ has quotient singularities, there is a canonical way to define a smooth Deligne-Mumford stack $\Sstack_K$ together with a morphism
\begin{equation}\label{eq:orbi-res}
    \epsilon \colon \Sstack_K \longrightarrow S_K
\end{equation}
that is a bijection on points and exhibits $S_K$ as a moduli space for $\Sstack_K$ \cite[Proposition 2.8]{vistoli_intersection-theory-stacks_89}. The stack $\Sstack_K$ is constructed by glueing quotient stacks of \'etale-local quotient presentations of $S_K$. In particular, $\epsilon$ is an isomorphism over the non-singular points of $S_K$. We will also refer to it as an ``orbifold resolution''. At a singular point $x$ of $S_K$, $\Sstack_K$ has isotropy group $\Gamma_x$ which, through its action on the tangent space of $\Sstack_K$ at $x$, is isomorphic to the subgroup of $\SL(2)$ associated to the type of singularity at $x$.

Recall that we have a fixed isomorphism $\Reg_{S_K} \cong h_* \mathcal{V}_0 = e_0 h_* \Ealg \! e_0$. The idempotent $e_0$ therefore defines an exact cornering functor \[\Coh(h_* \Ealg) \longrightarrow \Coh(S_K) \; , \quad F \mapsto e_0 F \; .\]

\begin{theorem}\label{thm:nc-orbifold-equivalence}
    $\Sstack_K$ is Morita equivalent to $h_* \Ealg$ over $S_K$. More precisely, we have the following diagram, commutative up to natural isomorphism \[\begin{tikzcd}[column sep=small] \Coh(h_* \Ealg) \arrow[rr, equal, "\sim"] \arrow[rd, "e_0"'] && \Coh(\Sstack_K) \arrow[ld, "\epsilon_*"] \\ & \Coh(S_K) \end{tikzcd}\]
\end{theorem}

We postpone the proof of Theorem \ref{thm:nc-orbifold-equivalence} to section \ref{sec:nc-orbifold-equivalence-proof}. Together with the derived equivalence of Theorem \ref{thm:derived-equivalence}, it extends the McKay correspondence as stated below. This result also follows from Chen--Tseng \cite{chen_tseng_derived-mckay_08}, once it is understood that $S$ can be identified with the correct component of $\Hilb(\Sstack_K)$, which we show in Corollary \ref{cor:quot-hilb-stack}(\ref{cor:quot-hilb-stack:2}) below.

\begin{corollary}[\cite{chen_tseng_derived-mckay_08}]\label{cor:nc-orbifold-equivalence}
    \[\Db(\Sstack_K) \cong \Db(S)\]
\end{corollary}

\begin{example}\label{ex:morita-preproj-equiv}
    In the case of $J=\{0\}$, the Morita equivalence of Theorem \ref{thm:nc-orbifold-equivalence} is that of Proposition \ref{thm:morita-skew-preproj}: The orbifold resolution of $\Sing$ is given by the quotient stack $[\Sing] \to \Sing$, and we have equivalences \[\Coh([\Sing]) \cong \Coh_{\Gamma}(\Aff^2) \cong \Mod(\C[x,y] \rtimes \Gamma) \cong \Mod(\Pi) \cong \Coh(f_* \Ealg) \; .\] 
\end{example}

\begin{example}\label{ex:morita-projective-mckay}
    Suppose $\Gamma<\SL(2,\C)$ contains the element $-1$. This is the case in types $A_n$ for $n$ odd, and in all types $D$ and $E$. Equivalently, $\Gamma$ is a degree-$2$ extension of its image $\Prj\!\Gamma$ in $\PSL(2,\C)$.

    Then $\Singa{\{\pm 1\}}$ has a singularity of type $A_1$. A crepant resolution is given by the cotangent bundle of the projective line, $T^{\vee}\!\Prj^1$. On the other hand, the action of $\Prj\!\Gamma$ on $\Prj^1$ gives rise to a symplectic action on $T^{\vee}\!\Prj^1$. We obtain the following commutative diagram, in which the vertical maps are crepant (partial) resolutions and the horizontal maps are quotients by $\Prj\!\Gamma$. \[\begin{tikzcd} T^{\vee}\!\Prj^1 \arrow[r] \arrow[d] & T^{\vee}\!\Prj^1\!\!/\Prj\!\Gamma \arrow[d] \\ \Singa{\{\pm 1\}} \arrow[r] & \Sing \end{tikzcd}\]

    Note that $T^{\vee}\!\Prj^1\!\!/\Prj\!\Gamma$ contains a single projective curve, which is the image of the zero section in $T^{\vee}\!\Prj^1$. In fact, we have an isomorphism \[T^{\vee}\!\Prj^1\!\!/\Prj\!\Gamma \cong S_K\] over $\Sing$, where $K = I \setminus \{0,r\}$, and $r$ is given as follows.
    \begin{itemize}
        \item In type $A_n$, $n$ odd, $r$ corresponds to the curve in the middle of the exceptional divisor on $S$.

        \item In types $D_n$, $E_6$, $E_7$, $E_8$, $r$ corresponds to the unique trivalent vertex of the finite-type Dynkin diagram.
    \end{itemize}
    The orbifold resolution of $S_K$ is then given by the quotient stack \[ \Sstack_K = [T^{\vee}\!\Prj^1\!\!/\Prj\!\Gamma] \; . \] The derived equivalence of Corollary \ref{cor:nc-orbifold-equivalence} in this case reads \[\Dba{\Prj\!\Gamma}(T^{\vee}\!\Prj^1) \cong \Db(S) \; .\] This (or rather, the derived equivalence between $[T^{\vee}\!\Prj^1\!\!/\Prj\!\Gamma]$ and $\Pi$) was observed and described under the name ``projective McKay correspondence'' by Brav \cite{brav_projective-mckay_09}.
\end{example}

\begin{example}\label{ex:morita-normal-subgp}
    The previous example can be generalised as follows (cf.\ \cite{ishii_ito_de-celis_gn-n-hilb_13}). Suppose $\Gamma$ contains a non-trivial, proper, normal subgroup $N \triangleleft \Gamma$. Then we have a morphism between singularities $\Singa{N} \to \Sing$ which is the quotient map for the induced action of $H \coloneqq \Gamma/N$ on $\Singa{N}$. Let $Y$ be the (classical) crepant resolution of $\Singa{N}$. The action of $H$ on $\Singa{N}$ lifts to an action on $Y$. Then we have a crepant partial resolution $Y\!/H \to \Sing$, its orbifold resolution is given by $[Y\!/H] \to Y\!/H$, and Corollary \ref{cor:nc-orbifold-equivalence} reads \[\Dba{H}(Y) \cong \Db(S) \; .\]
\end{example}

\begin{remark}
    In recent work, van de Kreeke \cite{vandekreeke_resolutions-tannaka_25} (building on an idea of Abdelgadir--Segal \cite{abdelgadir_segal_mckay-d4_24}) uses an approach inspired by Tannaka duality to construct the resolutions $S$ and $[\Sing]$ as quotients of the same algebraic group action for different stability. It is an interesting question whether this group action (or a similar one) yields all the resolutions $\Sstack_K$ under varying stability conditions.
\end{remark}

\subsection{Surfaces with ADE singularities and full reflexive sheaves}\label{sec:global-ade-full-reflexive}

Before proving Theorem \ref{thm:nc-orbifold-equivalence}, we will generalise some results from section \ref{sec:resolutions}. Let $X$ be a surface with ADE singularities, i.e., a surface that is smooth away from a set of isolated points, at each of which the completion of $X$ is isomorphic to $\widehat{\Sing}$ for some $\Gamma < \SL(2)$.

\begin{definition}
    A coherent sheaf on $X$ is a \emph{full reflexive sheaf} if it is reflexive and its restriction to the completion $\hat{X}_x$ at any closed point $x \in X$ contains all the indecomposable reflexive sheaves (as determined by Theorem \ref{thm:mckay}(\ref{thm:mckay:3})) over $\hat{X}_x$ as direct summands.
\end{definition}

\begin{example}
    The partial resolution $S_K$ is a surface with ADE singularities, and the sheaf $h_* \mathcal{V}$ on $S_K$ is full reflexive. Indeed, we already saw in Lemma \ref{lem:push-taut-prop}(\ref{lem:push-taut-prop:2}) that $h_* \mathcal{V}$ is reflexive, and following the argument presented there we see that, by Theorem \ref{thm:mckay}(\ref{thm:mckay:1}), the indecomposable reflexive sheaves appearing as direct summands in $h_* \mathcal{V}$ at a point $x$ correspond precisely to the curves in $h^{-1}(x)$ along which $\mathcal{V}$ has non-zero degree, which is the case along every such curve.
\end{example}

\begin{lemma}\label{lem:full-reflexive-morita}
    Let $X$ be a surface with ADE singularities and $\mathcal{W}$, $\mathcal{W}'$ be two full reflexive sheaves on $X$. Then $\SEnd_X(\mathcal{W})$ and $\SEnd_X(\mathcal{W}')$ are Morita equivalent via the same functors as in Theorem \ref{thm:morita-general}.
\end{lemma}
\begin{proof}
    The proof proceeds in the same way as the proof of Theorem \ref{thm:morita-general}, the only difference being that it suffices to check that the homomorphism \[\SHom_X(\mathcal{W}', \mathcal{W}) \otimes_{\SEnd_X(\mathcal{W}')} \SHom_X(\mathcal{W}, \mathcal{W}') \longrightarrow \SEnd_X(\mathcal{W})\] and its opposite are isomorphisms complete-locally instead of Zariski-locally.
\end{proof}

\begin{lemma}\label{lem:global-quotient-reflexive}
    Let $Y$ be a smooth, quasi-projective surface, and let $H$ act on $Y$ in such a way that any $h \in H$ stabilising a point $y \in Y$ acts on the tangent space at $y$ with determinant $1$. Let $X \coloneqq Y\!/H$ be the quotient variety and $\pi_X \colon Y \to X$ be the quotient map. Then
    \begin{enumerate}
        \item\label{lem:global-quotient-reflexive:1} $X$ is a surface with ADE singularities;
        \item\label{lem:global-quotient-reflexive:2} $\pi_{X,*}\Reg_Y$ is a full reflexive sheaf on $X$;
        \item\label{lem:global-quotient-reflexive:3} The natural homomorphism $(\pi_{X,*} \Reg_Y) \rtimes H \to \SEnd_X(\pi_{X,*}\Reg_Y)$ of coherent sheaves of algebras on $X$ is an isomorphism.
    \end{enumerate}
\end{lemma}
\begin{proof}
    \begin{enumerate}
        \item This is clear.

        \item By the same argument as in the proof of Lemma \ref{lem:push-taut-prop}(\ref{lem:push-taut-prop:2}), it suffices to show that $\pi_{X,*} \Reg_Y$ is full reflexive complete-locally around every closed point of $X$. Let $x \in X$ be a closed point and consider the following Cartesian diagram.
        \[\begin{tikzcd}
            \hat{Y}_x \arrow[r] \arrow[d, "\hat{\pi}_{X,x}"] & Y \arrow[d, "\pi_X"] \\
            \hat{X}_x \arrow[r] & X
        \end{tikzcd}\]
        By flat base change, $\pi_{X,*} \Reg_Y|_{\hat{X}_x} \cong \hat{\pi}_{X,x,*} \Reg_{\hat{Y}_x}$. Let $y \in \pi^{-1}(x)$, and let $\hat{Y}_y$ be its connected component in $\hat{Y}_x$. Then the restriction of $\hat{\pi}_{X,x}$ to $\hat{Y}_y \to \hat{X}_x$ is isomorphic to the completion of a Kleinian quotient map, under which the structure sheaf is pushed forward to a full reflexive sheaf by Theorem \ref{thm:mckay}(\ref{thm:mckay:3}). Hence, \[\hat{\pi}_{X,x,*} \Reg_{\hat{Y}_x} = \bigoplus_{y \in \pi^{-1}(x)} \hat{\pi}_{X,x,*} \Reg_{\hat{Y}_y}\] is full reflexive.
        
        \item Again, we check that the homomorphism is an isomorphism complete-locally on $X$. Let $x \in X$, $y \in \pi^{-1}(x)$ and $\hat{Y}_y \subseteq \hat{Y}_x$ be as above. Then
        \begin{equation}\label{eq:auslander-complete-local}
            (\hat{\pi}_{X,x,*}\Reg_{\hat{Y}_y}) \rtimes \Stab_H(y) \longrightarrow \SEnd_{\hat{X}_x}(\hat{\pi}_{X,x,*}\Reg_{\hat{Y}_y})
        \end{equation}
        is an isomorphism by identifying $\hat{Y}_y \to \hat{X}_x$ with a Kleinian quotient and observing that Theorem \ref{thm:auslander} still holds after completion at $o$. Now, let $h_1, \ldots, h_N$ be representatives of the left cosets of $\Stab_H(y)$ in $H$. Then \[(\hat{\pi}_{X,x,*}\Reg_{\hat{Y}_x}) \rtimes H = \bigoplus_{k,l} h_k \left((\hat{\pi}_{X,x,*}\Reg_{\hat{Y}_y}) \rtimes \Stab_H(y)\right) h_l^{-1}\] and \[\SEnd_{\hat{X}_x}(\hat{\pi}_{X,x,*}\Reg_{\hat{Y}_x}) \cong \bigoplus_{k,l} h_k \left(\SEnd_{\hat{X}_x}(\hat{\pi}_{X,x,*}\Reg_{\hat{Y}_y}) \right)h_l^{-1} \; .\] Hence, the isomorphism (\ref{eq:auslander-complete-local}) extends to an isomorphism \[(\hat{\pi}_{X,x,*}\Reg_{\hat{Y}_x}) \rtimes H \longrightarrow \SEnd_{\hat{X}_x}(\hat{\pi}_{X,x,*}\Reg_{\hat{Y}_x}) \; ,\] which is the completion of $(\hat{\pi}_{X,*}\Reg_Y) \rtimes H \to \SEnd_X(\Reg_Y)$ at $x$. \qedhere
    \end{enumerate}
\end{proof}

\subsection{Proof of Theorem \ref{thm:nc-orbifold-equivalence}}\label{sec:nc-orbifold-equivalence-proof}

First, we establish an analogue of Theorem \ref{thm:nc-orbifold-equivalence} for quotient stacks, along with some conditions on compatibility with \'etale pullback. We will use this to prove Theorem \ref{thm:nc-orbifold-equivalence} by writing $\Sstack_K$ \'etale-locally over $S_K$ as a quotient stack and establishing the diagram of functors for categories of descent data for that covering.

\begin{lemma}\label{lem:nc-orbifold-equivalence}
    \begin{enumerate}
        \item\label{lem:nc-orbifold-equivalence:1} Let $V$ be a smooth, affine surface, $H$ be a finite group acting on $V$ subject to the condition in Lemma \ref{lem:global-quotient-reflexive}, $U \coloneqq V\!/H$ be the quotient variety with ADE singularities and $\epsilon_U \colon [V\!/H] \to U$ be the natural map. Further, let $\mathcal{W}$ be a full reflexive sheaf on $U$ and write $\Ealg_U \coloneqq \SEnd_U(\mathcal{W})$. Then we have the following diagram. \[\begin{tikzcd} \Coh(\Ealg_U) \arrow[rd, "\Omega_V", "\sim"'] \arrow[dd, "\Omega_U"{name=DD, left}] & \\ & \Coh([V\!/H]) \arrow[ld, "\epsilon_{U,*}"] \arrow[from=DD, Rightarrow, "\omega_U"', "\sim", shorten <=3pt] \\ \Coh(U) \end{tikzcd}\] Here, $\Omega_V$ is an equivalence, and $\omega_U \colon \Omega_U \to \epsilon_{U,*} \circ \Omega_V$ is a natural isomorphism.

        \item\label{lem:nc-orbifold-equivalence:2} Let $(V,H,U,\epsilon_U)$ and $(V',H',U',\epsilon_{U'})$ each be as in (\ref{lem:nc-orbifold-equivalence:1}), and let \[\begin{tikzcd}[row sep=small] V' \arrow[d] \arrow[r, "v"] & V \arrow[d] \\ \lbrack V'\!/H'\rbrack \arrow[r, "v"] \arrow[d, "\epsilon_{U'}"] & \lbrack V\!/H\rbrack \arrow[d, "\epsilon_{U}"] \\ U' \arrow[r, "u"] & U \end{tikzcd}\] be a commutative diagram in which $v$ is equivariant with respect to some group homomorphism $H' \to H$ and the horizontal morphisms are \'etale. Then we have natural isomorphisms $\psi_v$, $\psi_u$ making the following diagrams $2$-commute. \[\begin{tikzcd}[row sep=large]
            \Coh(u^* \Ealg_U) \arrow[d, "\Omega_{V'}"] & \Coh(\Ealg_U)\arrow[l, "u^*"']\arrow[d, "\Omega_V"] \\ \Coh([V'\!/H']) & \Coh([V\!/H]) \arrow[l, "v^*"']\arrow[lu, Rightarrow, "\psi_v"']
        \end{tikzcd} \quad \text{and} \quad \begin{tikzcd}[row sep=large]
            \Coh(u^* \Ealg_U) \arrow[d, "\Omega_{U'}"] & \Coh(\Ealg_U)\arrow[l, "u^*"']\arrow[d, "\Omega_U"] \\ \Coh(U') & \Coh(U) \arrow[l, "u^*"']\arrow[lu, Rightarrow, "\psi_u"']
        \end{tikzcd}\]
        Furthermore, $\psi_u$ and $\psi_v$ are compatible with $\omega_U$ and $\omega_{U'}$ in that we have the following commutative diagram.
        \[\begin{tikzcd}[column sep=large]
            u^* \Omega_U \arrow[r, "u^* \omega_U"]  \arrow[dd, "\psi_u"] & u^* \epsilon_{U,*} \Omega_V \arrow[d, "\sim"]\\
            & \epsilon_{U', *} v^* \Omega_V \arrow[d, "\epsilon_{U', *} \psi_v"] \\
            \Omega_{U'}u^* \arrow[r, "\omega_{U'} u^*"] & \epsilon_{U',*} \Omega_{V'} u^*
        \end{tikzcd}\]
        In particular, the homomorphism in the top-right corner, which is the natural base change homomorphism, is an isomorphism.

        \item\label{lem:nc-orbifold-equivalence:3} Let $(V,H,U,\epsilon_U)$, $(V',H',U',\epsilon_{U'})$ and $(V'',H'',U'',\epsilon_{U''})$ each be as in (\ref{lem:nc-orbifold-equivalence:1}), and let \[\begin{tikzcd}[row sep=small] V'' \arrow[r, "v'"] \arrow[d] & V' \arrow[r, "v"] \arrow[d] & V \arrow[d] \\ \lbrack V''\!/H''\rbrack \arrow[r, "v'"] \arrow[d, "\epsilon_{U''}"] & \lbrack V'\!/H'\rbrack \arrow[r, "v"] \arrow[d, "\epsilon_{U'}"] & \lbrack V\!/H\rbrack \arrow[d, "\epsilon_{U}"] \\ U'' \arrow[r, "u'"] & U' \arrow[r, "u"] & U \end{tikzcd}\] be a diagram where both halves are as in (\ref{lem:nc-orbifold-equivalence:2}). Then the $\psi_v$, $\psi_{v'}$, $\psi_{v \circ v'}$ are compatible, in that the following diagram commutes.
        \[\begin{tikzcd}[column sep=large]
            (v \circ v')^*\Omega_V \arrow[rr, Rightarrow, "\psi_{v \circ v'}"] \arrow[d, equal, "\sim"] && \Omega_{V''}(u \circ u')^* \arrow[d, equal, "\sim"] \\
            v'^*v^*\Omega_V \arrow[r, Rightarrow, "v'^* \psi_v"] & v'^* \Omega_{V'} u^* \arrow[r, Rightarrow, "\psi_{v'}u^*"] & \Omega_{V''}u'^*u^*
        \end{tikzcd}\]
        The analogous statement holds for $\psi_u$, $\psi_{u'}$, $\psi_{u \circ u'}$.
    \end{enumerate}
\end{lemma}
\begin{proof}
    As all of the maps involved in this Lemma are affine, we will identify all sheaves with their pushforwards to $U$.
    \begin{enumerate}
        \item Let $F \in \Coh(\Ealg_U)$. We define \[\Omega_V(F) \coloneqq \SHom_U(\mathcal{W}, \Reg_V) \otimes_{\Ealg_U} F \quad \text{and} \quad \Omega_U(F) \coloneqq \SHom_U(\mathcal{W}, \Reg_U) \otimes_{\Ealg_U} F \; ,\] where $\Omega_V(F)$ is considered an element in $\Coh([V\!/H])$ via its $\SEnd_U(\Reg_V)$- and therefore $H$-equivariant $\Reg_V$-module structure. To turn $\Omega_V$ and $\Omega_U$ into functors, they are defined on morphisms in the obvious way.

        To see that $\Omega_V$ is an equivalence, identify $\Coh([V\!/H])$ with $\Coh(\Reg_V \rtimes H)$ and note that \[\Reg_V \rtimes H \cong \SEnd_U(\Reg_V)\] by Lemma \ref{lem:global-quotient-reflexive}(\ref{lem:global-quotient-reflexive:3}). As $\mathcal{W}$ is full reflexive by assumption and $\Reg_V$ is full reflexive by Lemma \ref{lem:global-quotient-reflexive}(\ref{lem:global-quotient-reflexive:2}), the functor $\Omega_V$ is the Morita equivalence of Lemma \ref{lem:full-reflexive-morita}, Theorem \ref{thm:morita-general}.

        Finally, we describe the natural isomorphism $\omega_U$. The pushforward of $\Omega_V(F)$ to $U$ can be identified with the $H$-invariant sections in $\SHom_U(\mathcal{W}, \Reg_V) \otimes_{\Ealg_U} F$, on which the $H$-action is defined by that on $\Reg_V$. We clearly have \[\SHom_U(\mathcal{W}, \Reg_V)^H = \SHom_U(\mathcal{W}, \Reg_V^H) = \SHom_U(\mathcal{W}, \Reg_U) \; .\] Since $\SHom_U(\mathcal{W}, \Reg_U)$ is a flat right $\Ealg_U$-module, the natural homomorphism $\Omega_U(F) \to \Omega_V(F)$ is an injection, and its image is precisely the subspace of $H$-invariants.

        \item We describe the natural isomorphism $\psi_v$. The transformation $\psi_u$ is defined analogously. Given an element $F \in \Coh(\Ealg_U)$, we define $\psi_v(F)$ as the composition of natural isomorphisms
        \[\begin{split}
            v^*\Omega_V(F) &= \Reg_{V'} \otimes_{\Reg_V} \SHom_U(\mathcal{W}, \Reg_V) \otimes_{\Ealg_U} F \\
            & \cong \Reg_{V'} \otimes_{\Reg_V} \SHom_{\Reg_V}(\mathcal{W} \otimes_U \Reg_V, \Reg_V) \otimes_{\Ealg_U} F \\
            & \cong \SHom_{\Reg_{V'}}(\mathcal{W} \otimes_U \Reg_{V'}, \Reg_{V'}) \otimes_{\Ealg_U} F \\
            & \cong \SHom_{\Reg_{V'}}(\mathcal{W} \otimes_U \Reg_{V'}, \Reg_{V'}) \otimes_{\Reg_{U'} \otimes_U \Ealg_U} (\Reg_{U'} \otimes_U \Ealg_U) \otimes_{\Ealg_U} F \\
            & \cong \SHom_{\Reg_{V'}}(\mathcal{W} \otimes_U \Reg_{V'}, \Reg_{V'}) \otimes_{\Reg_{U'} \otimes_U \Ealg_U} (\Reg_{U'} \otimes_U F) \\
            & \cong \SHom_{\Reg_{U'}}(\mathcal{W} \otimes_U \Reg_{U'}, \Reg_{V'}) \otimes_{\Reg_{U'} \otimes_U \Ealg_U} (\Reg_{U'} \otimes_U F) = \Omega_{V'}u^*(F)
        \end{split}\]
        Here, we used the fact that $\SHom$ commutes with flat pullback in the third line, as well as standard adjunctions and associativity of tensor products.

        The compatibility of $\psi_u$ and $\psi_v$ with $\omega_U$ follows immediately after unravelling their respective definitions.

        \item Likewise, this is straight-forward (although notationally exhausting) to check from the definitions. \qedhere
    \end{enumerate}
\end{proof}

\begin{proof}[Proof of Theorem \ref{thm:nc-orbifold-equivalence}]
    By \cite[Lemma 2.2.3]{abramovich_vistoli_compactifying_02},  there is an affine scheme $U$ together with an \'etale covering $U \to S_K$ over which $\Sstack_K$ is a quotient stack. More explicitly, there is a smooth, symplectic, affine scheme $V$, \'etale over $\Sstack_K$, together with a finite group $H$ acting symplectically on $V$ such that $U \cong V\!/H$ and the following square is Cartesian. \[\begin{tikzcd} \lbrack V\!/H\rbrack \arrow[r] \arrow[d] & \Sstack_K \arrow[d, "\epsilon"] \\ U \arrow[r] & S_K \end{tikzcd}\] (A priori, in \emph{loc.~cit.}, the finite group may vary across different connected components of the covering. One can make them the same by taking the product of all the groups and multiplying every connected component accordingly.) The \v{C}ech nerves of the respective covers, truncated after triple intersections, can be identified with the following diagram
    \begin{equation}\label{eq:cech-morita}
        \begin{tikzcd} \lbrack (V \times_{\Sstack_K} V \times_{\Sstack_K} V) / (H \times H \times H) \rbrack \arrow[r, shift left=6pt] \arrow[r] \arrow[r, shift right=6pt] \arrow[d] & \lbrack (V \times_{\Sstack_K} V) / (H \times H) \rbrack \arrow[r, shift left=3pt] \arrow[r, shift right=3pt] \arrow[d] & \lbrack V\!/H\rbrack \arrow[d] \\
        U \times_{S_K} U \times_{S_K} U \arrow[r, shift left=6pt] \arrow[r] \arrow[r, shift right=6pt] & U \times_{S_K} U \arrow[r, shift left=3pt] \arrow[r, shift right=3pt] & U \end{tikzcd}
    \end{equation}
    Here, $H \times H$ and $H \times H \times H$ again act symplectically and the vertical maps are coarse moduli spaces.

    For the rest of this proof, we shall write $V^k$ for the $k$-fold fibre product of $V$ over $\Sstack_K$, and $U^k$ for the $k$-fold fibre product of $U$ over $S_K$, respectively. Further, we shall denote horizontal morphisms on the right, from top to bottom, by $v_1$, $v_2$, $u_1$, and $u_2$, and the horizontal morphisms on the left, from top to bottom, by $v_{12}$, $v_{13}$, $v_{23}$, $u_{12}$, $u_{13}$, and $u_{23}$. 
    
    The functors introduced in Lemma \ref{lem:nc-orbifold-equivalence}(\ref{lem:nc-orbifold-equivalence:1}), together with pullback along the morphisms appearing in (\ref{eq:cech-morita}), form the following diagram.
    \[\begin{tikzcd}[column sep=tiny]
        \Coh(h_* \Ealg|_{U^3}) \arrow[rd] \arrow[dd] && \Coh(h_* \Ealg|_{U^2}) \arrow[dd, near end] \arrow[ll, shift left=6pt] \arrow[ll] \arrow[ll, shift right=6pt] \arrow[rd] && \Coh(h_* \Ealg|_U) \arrow[dd, near end] \arrow[ll, shift left=3pt] \arrow[ll, shift right=3pt] \arrow[rd] \\
        & \Coh(\lbrack V^3 \!/ H^3 \rbrack) \arrow[ld] && \Coh(\lbrack V^2 \!/ H^2 \rbrack) \arrow[ll, shift left=6pt, crossing over]  \arrow[ll, shift right=6pt, crossing over] \arrow[ll, crossing over] \arrow[ld] && \Coh(\lbrack V\!/H\rbrack) \arrow[ll, shift left=3pt, crossing over] \arrow[ll, shift right=3pt, crossing over] \arrow[ld] \\
        \Coh(U^3) && \Coh(U^2) \arrow[ll, shift left=6pt] \arrow[ll] \arrow[ll, shift right=6pt]  && \Coh(U) \arrow[ll, shift left=3pt] \arrow[ll, shift right=3pt]
    \end{tikzcd}\]
    Recall that a descent datum for $\Coh(h_* \Ealg)$ consists of an element \[F \in \Coh(h_* \Ealg|_U) \; ,\] together with an isomorphism in $\Coh(h_* \Ealg|_{U^2})$, \[u_1^* F \xlongrightarrow{\phi_F} u_2^* F \; ,\] subject to the cocycle condition, which is the commutativity of the following diagram in $\Coh(h_* \Ealg|_{U^3})$.
    \[\begin{tikzcd}[column sep=large]
        u_{12}^*u_1^*F \arrow[r, "u_{12}^* \phi_F"] \arrow[d, equal] & u_{12}^* u_2^* F = u_{23}^* u_2^* F \arrow[r, "u_{23}^* \phi_F"] & u_{23}^*u_2^*F \arrow[d, equal] \\
        u_{13}^*u_1^* F \arrow[rr, "u_{13}^* \phi_F"] && u_{13}^*u_2^*F
    \end{tikzcd}\] 
    A morphism of descent data $(F, \phi_F)$ and $(G, \phi_G)$ consists of a homomorphism $F \xrightarrow{\chi} G$ such that the square
    \[\begin{tikzcd}
        u_1^*F \arrow[r, "u_1^*(\chi)"] \arrow[d, "\phi_F"] & u_1^* G \arrow[d, "\phi_G"] \\
        u_2^*F \arrow[r, "u_2^*(\chi)"] & u_2^* G
    \end{tikzcd}\]
    commutes. Descent data and their morphisms are defined analogously for $\Coh(S_K)$ and $\Coh(\Sstack_K)$, where for the latter we need to replace $u$ by $v$. All three categories satisfy descent, meaning that $\Coh(h_* \Ealg)$, $\Coh(\Sstack_K)$ and $\Coh(S_K)$ are equivalent to their respective categories of descent data for the given \'etale covers. Hence, in order to construct the postulated diagram, it suffices to construct it for the respective categories of descent data.

    Now we argue that the functors $\Omega_V$ and $\Omega_U$ defined in Lemma \ref{lem:nc-orbifold-equivalence}, along with the various natural transformations and compatibilities, define functors between the various categories of descent data, which we will again denote by $\Omega_V$ and $\Omega_U$. We start with $\Omega_V$, which we first define on objects. Take a descent datum $(F, \phi_F)$ for $\Coh(h_* \Ealg)$, then we define descent data $\Omega_V(F, \phi_F)$ as consisting of the object $\Omega_V(F)$, given by Lemma \ref{lem:nc-orbifold-equivalence}(\ref{lem:nc-orbifold-equivalence:1}), together with the homomorphism on top making the following diagram commute.
    \[\begin{tikzcd}[column sep=large]
        v_1^*\Omega_V(F) \arrow[r] \arrow[d, "\sim"', "\psi_{v_1}(F)"] & v_2^*\Omega_V(F) \arrow[d, "\sim"', "\psi_{v_2}(F)"] \\ \Omega_{V^2}u_1^*(F) \arrow[r, "\Omega_{V^2}(\phi_F)", "\sim"'] & \Omega_{V^2}u_2^*(F)
    \end{tikzcd}\]
    Here, we used the natural transformation from Lemma \ref{lem:nc-orbifold-equivalence}(\ref{lem:nc-orbifold-equivalence:2}). One can now check using the compatibilities in Lemma \ref{lem:nc-orbifold-equivalence}(\ref{lem:nc-orbifold-equivalence:3}) that $\Omega_V(F, \phi_F)$ again satisfies the cocycle condition, thus making it a descent datum. Further, it is straight-forward to check using the naturality of $\phi_{v_1}$ and $\phi_{v_2}$ that $\Omega_V$ turns morphisms of descent data into morphisms of descent data, concluding the definition of the functor $\Omega_V$ for descent data. The definition of $\Omega_U$ for descent data is entirely analogous.

    Note that $\epsilon_*$, which is already defined as a functor $\Coh(\Sstack_K) \to \Coh(S_K)$, is defined as a functor of descent data as it commutes with pullback along \'etale covers. (This was shown in Lemma \ref{lem:nc-orbifold-equivalence}(\ref{lem:nc-orbifold-equivalence:2}), but can also be viewed as a version of flat base change for DM stacks, since the squares in (\ref{eq:cech-morita}) are Cartesian.) Using the compatibility between $\omega_U$, $\omega_{U'}$, $\psi_u$ and $\psi_v$ from Lemma \ref{lem:nc-orbifold-equivalence}(\ref{lem:nc-orbifold-equivalence:2}), one can see that the natural transformation $\omega_U$ induces a natural transformation of functors of descent data.

    Finally, since $e_0 h_* \mathcal{V} \cong \Reg_{S_K}$, we have an isomorphism, natural in $F$: \[\Omega_U(F) = \SHom(h_* \mathcal{V}|_U, \Reg_U) \otimes_{h_* \Ealg|_U} F \cong e_0 h_*\Ealg|_U \otimes_{h_* \Ealg|_U} F \cong e_0 F\] Hence, the functor defined by $\Omega_U$ is isomorphic to the cornering functor $e_0$.
\end{proof}

\begin{remark}
    Our proof of Theorem \ref{thm:nc-orbifold-equivalence} shows more generally that for any surface $X$ with ADE singularities, equipped with a full reflexive sheaf $\mathcal{W}$, the canonical stacky resolution $\mathscr{X} \to X$ is Morita equivalent to the noncommutative resolution $\SEnd_X(\mathcal{W})$ over $X$.
\end{remark}

\section{Hilbert schemes and quiver varieties}\label{sec:quot}

\subsection{Hilbert schemes of points}\label{sec:hilbs}

Recall that $\Hilb^n(X)$, the Hilbert scheme of $n$-points on a scheme $X$, is the fine moduli space of quotients $\Reg_X \twoheadrightarrow \Reg_Z$, where $\Reg_Z$ has zero-dimensional support and $\dim H^0(\Reg_Z)=n$. Since $S$ is a smooth surface, $\Hilb^n(S)$ is smooth by \cite[Theorem 2.4]{fogarty_families_68}.

For general $K \subseteq I \setminus \{0\}$, $\Hilb^n(S_K)$ is singular but irreducible, as was shown by Zheng \cite{zheng_irreducibility-hilb-kleinian_23}. It is an open question whether the natural scheme structure on $\Hilb^n(S_K)$ is always reduced (Craw--Yamagishi \cite{craw_yamagishi_hilb-can_26} show that it is for $n \leq 7$). As reducedness is necessary in the construction of the isomorphism at the bottom of Theorem \ref{thm:isom-quot-nqv}, we will henceforth equip $\Hilb^n(S_K)$ with its underlying reduced scheme structure.

The notion of a Hilbert scheme was generalised to DM stacks by Olsson--Starr \cite{olsson_starr_quot-dm_03}. Their results imply that there is a fine moduli space $\Hilb(\Sstack_K)$ whose connected components are quasi-projective schemes parametrising quotients $\Reg_{\Sstack_K} \twoheadrightarrow F$. 

\begin{definition}
    We define \[\Hilb^n(\Sstack_K) \subseteq \Hilb(\Sstack_K)\] to be the open and closed subscheme parametrising zero-dimensional quotients $\Reg_{\Sstack_K} \twoheadrightarrow F$, where
    \begin{enumerate}
        \item for each $\C$-point $x$ of $\Sstack_K$ with (possibly trivial) isotropy group $\Gamma_x$, and direct summand $F_x$ of $F$ supported at $x$, the $\Gamma_x$-representation underlying $F_x$ is isomorphic to $l_x$ times the regular representation for some non-negative integer $l_x$;

        \item with $l_x$ given as above, $\sum_{x \in S_K} l_x = n$.
    \end{enumerate}
\end{definition}

For $K=\emptyset$, the first condition is automatically satisfied for $l_x = \dim H^0(F_x)$. Hence, in this case, the component defined above is the usual Hilbert scheme of $n$-points on $S$. In general, the first condition is equivalent to requiring that the pullback of $F$ along any local quotient presentation \[U \to [U\!/\Gamma_x] \to \Sstack_K\] is a regular representation of $\Gamma_x$.

We also consider the following notion of an equivariant Hilbert scheme.

\begin{definition}
    Let $H$ be a finite group acting on a quasi-projective scheme $Y$. Let $n \in \N$. Define \[\GHilb{nH}(Y)\] to be the open and closed subscheme in the $H$-fixed point subscheme $\Hilb^{n|H|}(Y)^H \subset \Hilb^{n|H|}(Y)$ consisting of $H$-invariant subschemes $Z$ of $Y$ such that:
    \begin{enumerate}
        \item For any $H$-invariant open subset $U \subset Y$, $H^0(\Reg_Z|_U)$, as a representation of $H$, is isomorphic to a direct sum of regular representations;
        
        \item $H^0(\Reg_Z)$, as a representation of $H$, is isomorphic to $n$ times the regular representation.
    \end{enumerate}
\end{definition}

The first condition above is equivalent to requiring that the space of sections of the open and closed subscheme of $Z$ supported along any $H$-orbit is a multiple of the regular representation. It follows from the second condition in the case where there is at most one non-free $H$-orbit. In general, however, the first does not follow from the second condition and we need it to prove the following statement.

\begin{proposition}\label{prop:hilb-stack-equiv}
    Suppose we have a global quotient presentation $\Sstack_K=[Y\!/H]$ (as, for instance, in Example \ref{ex:morita-normal-subgp}). Then, identifying sheaves on $\Sstack_K$ with $H$-equivariant sheaves on $Y$ defines an isomorphism \[\Hilb^n(\Sstack_K) \cong \GHilb{nH}(Y) \; .\]
\end{proposition}
\begin{proof}
    The equivalence \[\Coh(\Sstack_K) = \Coh([Y\!/H]) = \Coh_{\Gamma}(Y)\] already identifies finite-length quotients of the structure sheaf over $\Sstack_K$ and finite-length $H$-equivariant quotients of the structure sheaf in $Y$. It will suffice to show that the local condition -- to be $l_x$ times the regular representation over a point $x \in S_K$ -- agrees in the two settings.

    Let $y \in Y$ be in the preimage of $x \in S_K$ and have stabiliser subgroup $\Stab_H(y) \subseteq H$. Then the orbit $Hy$ is isomorphic, as an $H$-set, to $H/\Stab_H(y)$. Hence, any $H$-equivariant sheaf $G_x$ supported on $Hy$ is induced from its component $G_y$ supported at $y$: \[H^0(G_x) \cong \C\!H \otimes_{\Stab_H(y)} H^0(G_y) \; .\] Let $F_x$ be the element of $\Coh(\Sstack_K)$ corresponding to $G_x$. Since $\Sstack_K=[Y\!/H]$, we have an identification $\Stab_H(y) \cong \Gamma_x$, and, under this identification, the $\Gamma_x$-representation underlying $F_x$ is $H^0(G_y)$. Hence, the claim follows from the fact that a $\Gamma_x$-representation is $l_x$ times the regular representation if and only if its induced $H$-representation is $l_x$ times the regular representation.
\end{proof}

\subsection{Noncommutative Quot schemes}\label{sec:quots}

Let $v \in \N^I$. Consider the quasi-projective scheme \[\Quot_J^v \coloneqq \Quot_{h_*\!\Ealg}^v(h_* \mathcal{V}) \; ,\] which is the fine moduli space for quotients $h_*\mathcal{V} \twoheadrightarrow F$ in the category $\Coh(h_* \Ealg)$ such that $F$ has zero-dimensional support and dimension vector $v$ (that is, $H^0(e_i F)$ has dimension $v_i$ for all $i \in I$). To see that such a scheme exists, one may identify $\Quot_J^v$ with an open and closed subscheme of the closed subscheme parametrising quotients over $h_* \Ealg$ inside the standard Quot scheme $\Quot_{\Reg_{S_K}}^{\sum_i v_i}(h_* \mathcal{V})$.

We will be mostly interested in the case where $v$ is of the form $n\delta$ for some $n \in \N$. In this case, we have the following basic result.
\begin{lemma}\label{lem:dim-locally-multiple-delta}
    Let $h_*\mathcal{V} \twoheadrightarrow F$ be a quotient corresponding to a point in $\Quot_J^{n \delta}$. Let \[F = \bigoplus_{x \in S_K} F_x\] be the decomposition of $F$ into components $F_x$, each supported at a single point $x \in S_K$. Then each $F_x$ has dimension vector a multiple of $\delta$.
\end{lemma}
\begin{proof}
     Fix $x \in S_K$. Let $K_x \subseteq K \subset I$ be the (possibly empty) set of curves in $S$ that are mapped to $x$ by $h$ and set $J_x \coloneqq I \setminus K_x$. Then $h_* \mathcal{V}_{J_x} = \bigoplus_{i \in J_x} h_* \mathcal{V}_i$ is locally-free around $x$. Therefore, we have \[e_{J_x} F_x = h_* \mathcal{V}_{J_x} \otimes \SHom_{h_* \Ealg_{J_x}}(\mathcal{V}_{J_x}, e_{J_x} F_x) \; ,\] where $\SHom_{h_* \Ealg_{J_x}}(\mathcal{V}_{J_x}, e_{J_x} F_x)$ is a $\Reg_{S_K}$-module of length $l_x$. Hence, for $i \in J_x$, \[\dim H^0(e_iF_x) = \dim H^0(h_* \mathcal{V}_i \otimes \SHom_{h_* \Ealg_{J_x}}(\mathcal{V}_{J_x}, e_{J_x} F_x)) = l_x \delta_i \; .\]

    On the other hand, for $i \in K_x$, $h_* \mathcal{V}_i$ is locally-free around all other points $x' \neq x$ in $S_K$. We repeat the argument above for every such point and obtain \[\dim(H^0(e_iF_x)) = \dim H^0(e_iF) - \dim H^0\left(\bigoplus_{x' \neq x} e_i F_{x'}\right) = \left(n -\sum_{x' \neq x} l_{x'}\right) \delta_i = l_x \delta_i \; ,\] as $l_x = \dim H^0(e_0 F_x)$ and so $\sum_{x \in S_K} l_x = n$.
\end{proof}

\begin{proposition}\label{prop:quot-hilb-stack}
    The Morita equivalence of Theorem \ref{thm:nc-orbifold-equivalence} realises $\Quot_J^v$ as an open and closed subscheme of $\Hilb(\Sstack_K)$. For $v=n\delta$, this subscheme is precisely $\Hilb^n(\Sstack_K)$, and we have the following commutative diagram.
    \begin{equation}\label{eq:quot-to-hilb}
        \begin{tikzcd}[column sep=tiny]
        \Quot_J^{n\delta} \arrow[rr, "\sim"]\arrow[dr, "e_0"'] && \Hilb^n(\Sstack_K) \arrow[dl, "\epsilon_*"] \\
        & \Hilb^n(S_K) &
    \end{tikzcd}
    \end{equation}
    Both morphisms to $\Hilb^n(S_K)$ are crepant resolutions of singularities.
\end{proposition}
\begin{proof}
    The Morita equivalence of Theorem \ref{thm:nc-orbifold-equivalence} maps $h_* \mathcal{V}$ to the structure sheaf and preserves quotients and flat families. Since the dimension vector is locally constant in such families, the functor represented by $\Quot_J^v$ is isomorphic to an open and closed subfunctor of the functor represented by $\Hilb(\Sstack_K)$.
    
    In the special case of $v=n\delta$, it was shown in Lemma \ref{lem:dim-locally-multiple-delta} that any quotient $h_*\mathcal{V} \twoheadrightarrow F$ corresponding to a point in $\Quot_J^{n\delta}$ decomposes as a direct sum $F = \bigoplus_{x \in S_K} F_x$, where $F_x$ is supported at $x$ and has dimension vector $l_x \delta$. Under the Morita equivalence of Theorem \ref{thm:nc-orbifold-equivalence}, the component $F_x$ corresponds to a sheaf on $\Sstack_K$ supported on $x$ whose underlying $\Gamma_x$-representation is isomorphic to $l_x$ times the regular representation. This follows from the fact that in Theorem \ref{thm:morita-general}, the rank vectors of the reflexive sheaves whose endomorphism algebras are shown to be Morita equivalent transform in the same way as dimension vectors of finite-length sheaves. This implies that the image of $\Quot_J^{n \delta}$ in $\Hilb(\Sstack_K)$ is precisely $\Hilb^n(\Sstack_K)$.
    
    Similarly, the functors $e_0$ and $\epsilon_*$ from Theorem \ref{thm:nc-orbifold-equivalence} preserve flat families, and $e_0$ maps $h_* \mathcal{V}$ to $\Reg_{S_K}$ and a quotient of dimension vector $n \delta$ to a quotient of length $n$. Hence, we obtain the desired diagram, and it commutes because the underlying functors commute up to natural isomorphism.

    Finally, we show that the morphism $\Hilb^n(\Sstack_K) \to \Hilb^n(S_K)$ is a crepant resolution. The statement is \'etale-local on $\Hilb^n(S_K)$, so it will suffice to prove it after base change to an \'etale covering. Consider a point $z$ of $\Hilb^n(S_K)$, representing a subscheme of $S_K$ of length $l_x$ at the point $x$. Let $\Gamma_x < \SL(2)$ be the group associated with the ADE singularity at $x$ (or $\{1\}$ whenever $S_K$ is smooth at $x$). Then, using a local quotient presentations of $\Sstack_K$ and Theorem \ref{thm:nc-orbifold-equivalence}, \'etale-locally in $\Hilb^n(S_K)$ around $z$, the morphism (\ref{eq:quot-to-hilb}) is isomorphic to a product of quotient-scheme-morphisms (in terminology of \cite{brion_invariant-hilb_13}) \[\GHilb{l_x\Gamma_x}(\Aff^2) \longrightarrow \Hilb^{l_x}(\Singa{\Gamma_x}) \; .\] All these morphisms are crepant resolutions by \cite[Theorem 1.1]{cggs_punctual-hilb_21}, and so is their product.
\end{proof}

\begin{corollary}\label{cor:quot-hilb-stack}
    \ \begin{enumerate}
        \item\label{cor:quot-hilb-stack:1} For $n \in \N$, $\Quot_J^{n \delta}$ is birational to $\Hilb^n(S)$.
        
        \item\label{cor:quot-hilb-stack:2} $\Quot_J^{\delta} \cong \Hilb^1(\Sstack_K) \cong S$.
    \end{enumerate}
\end{corollary}
\begin{proof}
    \begin{enumerate} 
        \item Both $\Hilb^n(\Sstack_K) \cong \Quot_J^{n\delta}$ and $\Hilb^n(S) \cong \Quot_I^{n\delta}$ contain as a non-empty open subscheme the locus parametrising modules supported on the shared open subset $S \setminus \{f^{-1}(o)\} \cong S_K \setminus \{g^{-1}(o)\}$.

        \item For $n=1$, Proposition \ref{prop:quot-hilb-stack} states that $\Quot_J^{\delta} \cong \Hilb^1(\Sstack_K)$ is a crepant resolution of $\Hilb^1(S_K) \cong S_K$, hence, it is isomorphic to the unique crepant resolution $S$ of $S_K$. \qedhere
    \end{enumerate} 
\end{proof}

\subsection{Nakajima quiver varieties}\label{sec:nqvs}

Let $\Bar{\Pi}$ be the framed preprojective algebra associated to the affine ADE Dynkin diagram with vertex set $I$, with one framing vertex named $\infty$ connected to $0$ via two opposing arrows (see Figure \ref{fig:framed-quiver}). The underlying vertex set of $\Bar{\Pi}$ is denoted $\Bar{I} = \{\infty\} \cup I$. The unframed preprojective algebra $\Pi$ can be obtained as the quotient of $\Bar{\Pi}$ by the two-sided ideal generated by $e_{\infty}$, the idempotent associated with the framing vertex.

\begin{figure}
    \centering
    \begin{tikzpicture}[framing/.style={rectangle,draw=black,thick,inner sep=0pt,minimum size=5mm},ivert/.style={circle,draw=black,thick,inner sep=0pt,minimum size=5mm}]
        \node (f) at (1.5,.8) [framing] {$\infty$};
        \node (tl) at (3,.8) [ivert] {$0$};
        \node (bl) at (3,-.8) [ivert] {};
        \node (l) at (4,0) [ivert] {};
        \node (r) at (5.3,0) [ivert] {};
        \node (tr) at (6.3,.8) [ivert] {};
        \node (br) at (6.3,-.8) [ivert] {};
        
        \draw[->] (f) to [bend left=20] (tl);
        \draw[->] (tl) to [bend left=20] (f);
        \draw[->] (tl) to [bend left=20] (l);
        \draw[->] (l) to [bend left=20] (tl);
        \draw[->] (bl) to [bend left=20] (l);
        \draw[->] (l) to [bend left=20] (bl);
        \draw[->] (l) to [bend left=20] (r);
        \draw[->] (r) to [bend left=20] (l);
        \draw[->] (tr) to [bend left=20] (r);
        \draw[->] (r) to [bend left=20] (tr);
        \draw[->] (br) to [bend left=20] (r);
        \draw[->] (r) to [bend left=20] (br);
    \end{tikzpicture}
    \caption{Quiver underlying the framed preprojective algebra $\Bar{\Pi}$ of type $D_5$.}
    \label{fig:framed-quiver}
\end{figure}
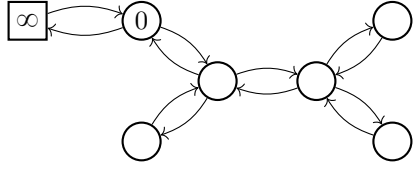

For $v \in \N^I$, we define the space of stability vectors \[\Theta_v \coloneqq \left\{\theta \in \Q^{\Bar{I}} \; \big| \; \theta \cdot (1,v) = 0 \right\} \, .\] Let $\theta \in \Theta_v$, and let $\Bar{V}$ be a $\Bar{\Pi}$-module of dimension vector $(1,v)$. Then $\Bar{V}$ is called $\theta$-semistable if, for any submodule $\Bar{V}' \subseteq \Bar{V}$, \[\theta \cdot \dim(\Bar{V}') \geq 0 \; .\] $\Bar{V}$ is called $\theta$-stable if it is $\theta$-semistable and the above inequality is strict for proper, non-zero submodules. With respect to this stability condition, $\Bar{\Pi}$-modules admit Harder-Narasimhan filtrations and a notion of S-equivalence. We fix $\Lambda_0 = (1,0,\ldots,0)$. Then the Nakajima quiver variety \[\NQV_{\theta}(v,\Lambda_0)\] is a coarse moduli space for S-equivalence classes of $\theta$-semistable $\Bar{\Pi}$-modules of dimension $(1,v)$. In section \ref{sec:quiver-mods}, we will give more details on the construction of quiver varieties in general, of which $\NQV_{\theta}(v,\Lambda_0)$ are special cases.

The space $\Theta_v$ admits a wall-and-chamber structure, where $\theta$-stability and $\theta$-semistability are equivalent as long as $\theta$ is in the interior of a chamber, and $\NQV_{\theta}(v,\Lambda_0 ) \cong \NQV_{\theta'}(v,\Lambda_0 )$ for $\theta$ and $\theta'$ in the interior of the same chamber. For vectors of the form $v = n \delta$, the wall-and-chamber structure was determined by Bellamy--Craw \cite{bellamy_craw_birational-symplectic_20}. To state their result, we identify $\Theta_v$ with the Cartan subalgebra of the affine Lie algebra, and let $\Phi^+ \subset (\Theta_v)^{\vee}$ denote the set of positive roots for the finite-type Cartan subalgebra. The identification is done in such a way that $\theta(\alpha_i)$ is the $i$-th entry of $\theta \in \Theta_v$ (in terms of the definition of $\Theta_v$ as a subspace of $\Q^{\Bar{I}}$), where $\alpha_i$ is the root corresponding to the vertex $i$ of the Dynkin diagram, and $\theta(\alpha_i)$ denotes the natural pairing between $\theta \in \Theta_v$ and $\alpha_i \in \Phi^+ \subset (\Theta_v)^{\vee}$.

\begin{theorem}[Bellamy--Craw]\label{thm:nqv-walls}
    Consider the following set of hyperplanes in $\Theta_{n\delta}$: \[\mathcal{A} \coloneqq \{\delta^{\perp}\} \cup \{(m \delta \pm \alpha)^{\perp} \; | \; 0 \leq m < n, \; \alpha \in \Phi^+ \}.\] For any $\theta \in \Theta_{n\delta} \setminus \bigcup_{A \in \mathcal{A}} A$,  $\NQV_{\theta}(n \delta,\Lambda_0 )$ is a smooth, Hyperk\"ahler variety, projective over the affine variety $\NQV_0(n \delta, \Lambda_0 )$, and a fine moduli space for $\theta$-stable $\Bar{\Pi}$-modules.
\end{theorem}

A particular collection of chambers is defined by Craw--Wye \cite{craw_wye_hilb-crepant-partial_25-v2} as follows. Let $K \subseteq I \setminus \{0\}$, $J=I \setminus K$, and define \[C_K \coloneqq \left\{\theta \in \Theta_{n\delta} \; \Bigg| \; \begin{matrix*}[l] \theta(\delta\vert_J) > 0 \; , & \\ \theta(\alpha_j) > (n-1)\theta(\delta) & \text{for} \; j \in J \setminus \{0\} \; , \\ \theta(\alpha_k) > 0 & \text{for} \; k \in K \, , \end{matrix*} \right\} \] where $\delta\vert_J = \sum_{i \in J} \delta_i \alpha_i$, and \[\sigma_K \coloneqq \left\{\theta \in \Theta_{n\delta} \; \Bigg| \; \begin{matrix*}[l] \theta(\delta) \geq 0 \; , & \\ \theta(\alpha_j) \geq (n-1)\theta(\delta) & \text{for} \; j \in J \setminus \{0\} \; , \\ \theta(\alpha_k) = 0 & \text{for} \; k \in K \, , \end{matrix*} \right\} \] which is a component of the boundary of $C_K$. (Note that in the definition of $\sigma_K$ it does not matter whether we write $\theta(\delta)$ or $\theta(\delta\vert_J)$.) All the $C_K$ and $\sigma_K$ are subsets of the simplicial cone \[F \coloneqq \left\{\theta \in \Theta_{n\delta} \; \Big| \; \begin{matrix*}[l] \theta(\delta) \geq 0 \; , & \\ \theta(\alpha_i) \geq 0 & \text{for} \; i \in I \setminus \{0\} \; . \end{matrix*} \right\}\] In Figure \ref{fig:walls}, we draw one example of the wall-and-chamber structure in a transversal slice of the cone $F$ and highlight the regions defined above .

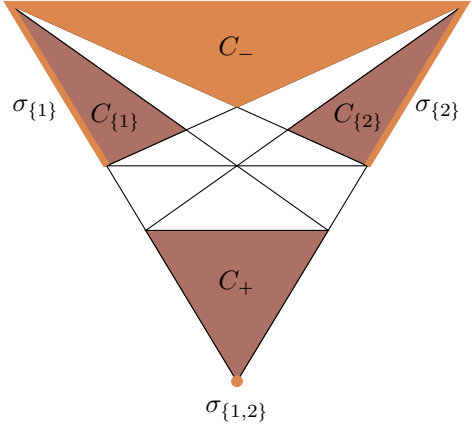
\begin{figure}[ht]
    \centering
    \begin{tikzpicture}
        \draw (0,0) -- (-3,5) -- (3,5) -- (0,0);
        \draw (-3,5) -- (1.2,2) -- (-1.2,2) -- (3,5);
        \draw (-3,5) -- (12/7,20/7) -- (-12/7,20/7) -- (3,5);
        
        \filldraw[fill=burntumber!70] (0,0) -- (-6/5,2) -- (6/5,2) -- cycle;
        \filldraw[fill=burntumber!70] (-3,5) -- (-2/3,10/3) -- (-12/7,20/7) -- cycle;
        \filldraw[fill=burntumber!70] (3,5) -- (2/3,10/3) -- (12/7,20/7) -- cycle;
        \filldraw[burntorange!70] (-3,5) -- (0,40/11) -- (3,5) -- cycle;
        \draw[draw=burntorange!70,line width=2.5pt] (-12/7,20/7) -- (-3,5) -- (3,5) -- (12/7,20/7);
        \filldraw[burntorange!70] (0,0) circle [radius=2pt];
        
        \draw (0,4.4) node {$C_-$};
        \draw (-1.6,3.5) node {$C_{\{1\}}$};
        \draw (-2.65,3.6) node {$\sigma_{\{1\}}$};
        \draw (1.6,3.5) node {$C_{\{2\}}$};
        \draw (2.65,3.6) node {$\sigma_{\{2\}}$};
        \draw (0,1.3) node {$C_+$};
        \draw (0,-0.4) node {$\sigma_{\{1,2\}}$};
    \end{tikzpicture}
    \caption{Transversal slice of the wall-and-chamber structure on the fundamental cone $F$ in type $A_2$, for $n=3$. Here, $I=\{0,1,2\}$, $C_+ = C_{I \setminus \{0\}}$, $C_- = C_{\emptyset}$, and $\sigma_{\emptyset}$ is the closure of $C_-$.}
    \label{fig:walls}
\end{figure}

The chamber $C_{\emptyset}$ is also referred to as $C_-$. A result proved by Kuznetsov \cite{kuznetsov_quiver-hilb_07} (and found independently by Haiman and Nakajima) establishes an isomorphism 
\begin{equation}
\label{eq:NQVforC-}
    \NQV_{\theta}(n \delta,\Lambda_0 ) \cong \Hilb^n(S) \quad \text{for} \quad \theta \in C_- \; .
\end{equation}
Recall also that $\Hilb^n(S) \cong \Quot_I^{n \delta}$. On the other hand, we have the chamber $C_+ \coloneqq C_{I \setminus \{0\}}$. For $\theta \in C_+$, a $\Bar{\Pi}$-module $\Bar{V}$ is $\theta$-stable if and only if it is generated by $e_{\infty} \Bar{V}$. By Theorem \ref{thm:nqv-walls}, $\NQV_{\theta}(n \delta,\Lambda_0 )$ is a fine moduli space for such modules, from which we deduce that 
\begin{equation}\label{eq:NQVforC+}
    \NQV_{\theta}(n \delta,\Lambda_0 ) \cong \Quot_{\{0\}}^{n \delta} \cong \Hilb^n([\Sing]) \quad \text{for} \quad \theta \in C_+ \; .
\end{equation}
Recent work of Craw--Gammelgaard--Gyenge--Szendr\H{o}i \cite{cggs_punctual-hilb_21} and Craw--Yamagishi \cite{craw_yamagishi_hilb-can_26} establishes an isomorphism between $\Hilb^n(\Sing)$ with the Nakajima quiver variety $\NQV_{\theta'}(n\delta, \Lambda_0)$ for $\theta'$ in the relative interior of $\sigma_{I \setminus \{0\}}$. Together with the isomorphism above, their isomorphism fits into the following commutative diagram.
\[\begin{tikzcd} \Hilb^n([\Sing]) \arrow[r, "\sim"] \arrow[d] & \NQV_{\theta}(n \delta,\Lambda_0 ) \arrow[d] \\ \Hilb^n(\Sing) \arrow[r, "\sim"] & \NQV_{\theta'}(n\delta, \Lambda_0 ) \end{tikzcd}\]

The following theorem contains a common generalisation of the isomorphisms in \eqref{eq:NQVforC-} and \eqref{eq:NQVforC+}. We postpone the proof to the following sections.
        
\begin{theorem}\label{thm:isom-quot-nqv}
    Let $I \supseteq J \ni 0$, $K \coloneqq I \setminus J$, $\theta_K \in C_K$, and $\theta'_K$ in the relative interior of $\sigma_K$. Then we have the following commutative diagram in which all morphisms are birational and projective, and the horizontal morphisms are isomorphisms. \[\begin{tikzcd} \Hilb^n(\Sstack_K) \arrow[r,"\nu", "\sim"'] \arrow[d,"\epsilon_*"] & \NQV_{\theta_K}(n \delta,\Lambda_0 ) \arrow[d] \\ \Hilb^n(S_K) \arrow[r, "\nu_J", "\sim"'] & \NQV_{\theta'_K}(n\delta, \Lambda_0 ) \end{tikzcd}\] Here, $\nu$ is induced, via the isomorphism of Proposition \ref{prop:quot-hilb-stack}, by the functor \[H^0 \colon \Coh(h_* \Ealg) \longrightarrow \Mod(\Pi) \; ,\] $\nu_J$ is induced by the functor \[\Hom(\mathcal{T}_J^{\vee},-) \colon \Coh(S_K) \to \Mod(\Pi_J) \; ,\] and the morphism on the right is a VGIT morphism.
\end{theorem}

\begin{corollary}
    The morphism $\Hilb^n(\Sstack_K) \to \Hilb^n(S_K)$ is the unique projective, crepant resolution of $\Hilb^n(S_K)$.
\end{corollary}
\begin{proof}
    According to \cite[Theorem 1.1]{craw_wye_hilb-crepant-partial_25-v2}, for $\theta'_K$ in the relative interior of $\sigma_K$, $\Hilb^n(S_K) \cong \NQV_{\theta'_K}(n\delta, \Lambda_0 )$ has a unique projective, crepant resolution given by $\NQV_{\theta_K}(n \delta, \Lambda_0)$. Thus, the statement follows from Theorem \ref{thm:isom-quot-nqv}.
\end{proof}

\begin{corollary}\label{cor:quot-from-nqv}
    Let $I \supseteq J \ni 0$, $n \geq 1$.
    \begin{enumerate}
        \item\label{cor:quot-from-nqv:1} The complex manifolds underlying $\Hilb^n(\Sstack_K)$ for varying $K$ are diffeomorphic.
        \item\label{cor:quot-from-nqv:2} $\Db(\Hilb^n(\Sstack_K)) \cong \Db(\Hilb^n(S))$.
    \end{enumerate}
\end{corollary}
\begin{proof}
    Both statements follow via Theorem \ref{thm:isom-quot-nqv} from the corresponding statements relating Nakajima quiver varieties for various generic stability conditions: see \cite[Corollary 4.2]{nakajima_instantons_94} for (\ref{cor:quot-from-nqv:1}), and \cite[Theorem 5.1, Proposition 5.6]{halpern-leistner_combinatorial-derived_20} for (\ref{cor:quot-from-nqv:2}).
\end{proof}

\begin{example}\label{ex:isom-eq-hilb-nqv}
    In Examples \ref{ex:morita-preproj-equiv}, \ref{ex:morita-projective-mckay}, \ref{ex:morita-normal-subgp}, where we have an explicit description of $\Sstack_K$ as a global quotient, say, \[\Sstack_K = [Y\!/H] \; ,\] combining Proposition \ref{prop:hilb-stack-equiv} and Theorem \ref{thm:isom-quot-nqv} gives isomorphisms \[\NQV_{\theta_K}(n \delta,\Lambda_0 ) \cong \GHilb{nH}(Y)\] for $\theta_K \in C_K$. In particular, Example \ref{ex:morita-projective-mckay} gives \[\NQV_{\theta_K}(n \delta,\Lambda_0 ) \cong \GHilb{n\Prj\!\Gamma}(T^{\vee}\!\Prj^1) \, .\]
\end{example}

\section{The isomorphism between the singular moduli spaces}\label{sec:geometric}

Here we construct explicitly the isomorphism $\Hilb^n(S_K)\to \NQV_{\theta'_K}(n\delta, \Lambda_0)$ from Theorem~\ref{thm:isom-quot-nqv}, which was shown to exist in \cite{craw_wye_hilb-crepant-partial_25-v2}. Our general approach follows that of Kuznetsov~\cite[Theorem~43]{kuznetsov_quiver-hilb_07} in the special case $K=\varnothing$. This explicit construction will play an important role in the proof of Theorem \ref{thm:isom-quot-nqv}, as much of the method will carry over to the Quot scheme case in section \ref{sec:proofsmooth}.

\subsection{Quiver moduli spaces}\label{sec:quiver-mods}

Analogously to the construction of the Nakajima quiver variety $\NQV_\theta(n\delta,\Lambda_0 )$ in section \ref{sec:nqvs}, we can consider quiver moduli spaces for different quivers and algebras. Here, we briefly summarise the construction and mention how the McKay correspondence and the partial McKay correspondence can be rephrased in terms of quiver moduli spaces. 

We consider algebras $A$ obtained as quotients of the path algebra of a quiver $\mathcal{Q}$ by an ideal of relations. Associated to this, and any dimension vector $u \in \N^{\mathcal{Q}_0}$, we have a space of stability conditions 
\[ \Theta_u = \{ \theta \in \QQ^{|\mathcal{Q}_0|} \mid \theta \cdot u =0\}\]
where $\mathcal{Q}_0$ is the set of vertices of the quiver. Let $\Rep(A,u)$ be the affine scheme of $A$-module structures on the $\mathcal{Q}_0$-graded vector space $\bigoplus_{i \in \mathcal{Q}_0} \C^{u_i}$. We can then define 
\[\mathcal{M}_\theta(A,u) \coloneqq \Rep(A,u) \git_\theta G(u)\]
as the GIT quotient of $\Rep(A,u)$ by the action of the group $G(u)= \left(\prod_{i \in \mathcal{Q}_0} \GL(u_i)\right)\!/\CC^*$ with character $\chi_\theta(g)=\prod_{i \in \mathcal{Q}_0} \det (g_i)^{\theta_i}$. This is the coarse moduli space for $\Seshadri$-equivalence classes of $\theta$-semistable $A$-modules of dimension $u$. By King \cite{King94}, this moduli space is a fine moduli space when $u$ is indivisible and $\theta$ is generic, meaning the notions of $\theta$-stability and $\theta$-semistability coincide. In section \ref{sec:nqvs}, we were concerned with the situation when $A= \overline{\Pi}$, the framed preprojective algebra, and then the quiver moduli space $\mathcal{M}_\theta(\overline{\Pi}, (1,n\delta))$ is the Nakajima quiver variety $\NQV_\theta(n\delta,\Lambda_0 )$.

Another quiver moduli space of interest is that for the unframed preprojective algebra $\Pi$ with dimension vector $\delta$. We fix a stability condition $\zeta \in \Theta_\delta$ by \begin{equation} \label{eqn:zeta}
       \zeta(\rho_i) = \left\{\begin{array}{cr} -\sum_{i \neq 0} \dim \rho_i & \text{ for }i = 0, \\ 1 & \text{otherwise.}\end{array}\right.
  \end{equation} 
Work of Kronheimer \cite{kronheimer_ale-hk-quotients_89} shows that $\mathcal{M}_\zeta(\Pi,\delta) \cong S$, and since $\zeta$ is generic and $\delta$ indivisible, this is a fine moduli space. The tautological bundle on this fine moduli space is the tilting bundle $\mathcal{V}$, providing an alternative viewpoint to the McKay correspondence of Theorem \ref{thm:mckay}.

We can also rewrite the partial McKay correspondence (Theorem \ref{thm:partial}) by writing $S_K$ as a fine quiver moduli space with tautological bundle $\mathcal{T}_J = h_* \mathcal{V}_J$. The algebra $\Pi_J$ can be obtained as the quotient of the path algebra of a finite quiver $Q_J$ by the kernel of relations $\ker \alpha_J$, as in the proof of Proposition 3.3 of \cite{cggs_punctual-hilb_21}. The nodes of $Q_J$ correspond to the irreducible representations $\rho_j$ for $j\in J$, so we fix a dimension vector $\delta\vert_J = \sum_{j \in J} \dim (\rho_j)\rho_j$, and group $G(\delta\vert_J)= \prod_{j \in J} \GL(\delta_j) /\CC^*$. We fix $\chi_J \in G(\delta\vert_J)^\vee$ by
\begin{equation} \label{eqn:chiJ}
       \chi_J(\rho_j) = \left\{\begin{array}{cr} -\sum_{j\in J \setminus 0} \delta_j & \text{ for }j = 0, \\ 1 & \text{otherwise.}\end{array}\right.
  \end{equation}
  Every $\chi_J$-semistable $\Pi_J$-module of dimension vector $\delta\vert_J$ is $\chi_J$-stable because the only nonzero vector
  $w\in \mathbb{N}^J$ satisfying $w\leq \delta\vert_J$ and $\chi_J(w)=0$ is $w=\delta\vert_J$. Since $\delta\vert_J$ is indivisible, \cite[Proposition~5.3]{King94} implies that $\mathcal{M}_{\chi_J}(\Pi_J,\delta\vert_J)$ is the fine moduli space of isomorphism classes of $\chi_J$-stable $\Pi_J$-modules of dimension vector $\delta\vert_J$.
Another way of writing Theorem \ref{thm:partial} is the following:

\begin{proposition}
\label{prop:PartialMcKayquiverver}
    The surface $S_K$ is isomorphic to $\mathcal{M}_{\chi_J}(\Pi_J, \delta\vert_J)$, a fine moduli space, with tautological bundle $\mathcal{T}_J = h_* \mathcal{V}_J$, which defines an equivalence of triangulated categories between $\Db(S_K)$ and $\Db(\Pi_J)$.
\end{proposition}

\subsection{Constructing a flat family}
 
Consider the universal subscheme $\mathcal{Z}_n \subseteq \Hilb^n(S_K) \times S_K$ and the morphisms
\[
\begin{tikzcd}
\mathcal{Z}_n \arrow[r, "q"] \arrow[d, swap,"p"] &  S_K \\
\Hilb^n(S_K) &  
\end{tikzcd}
\]
obtained by projecting onto each factor. For each closed point $y\in \Hilb^n(S_K)$, let $Z_y:= p^{-1}(y)$ denote the length $n$ subscheme of $S_K$ parametrised by $y$. 

As shown in Theorem~\ref{thm:partial}, the surface $S_K$ carries a tilting bundle of the form $\mathcal{T}_J:=\bigoplus_{j\in J}\mathcal{T}_j$ where $\mathcal{T}_0\cong \mathcal{O}_{S_K}$, and $\mathcal{T}_j$ is locally-free of rank $\dim\rho_j$ for each $j\in J$. Since $p$ is flat, finite with $p_* \Reg_{\mathcal{Z}_n}$ of rank $n$, the sheaf
\begin{equation}\label{eqn:TJ[n]}
    \mathcal{T}_j^{[n]}:= p_*\big(q^*(\mathcal{T}_j)\big)
\end{equation}
on $\Hilb^n(S_K)$ is locally-free of rank $n \delta_j$, and moreover, the fibre at a closed point $y \in \Hilb^n(S_K)$ is isomorphic to $H^0(Z_y, \Oo_{Z_y} \otimes \mathcal{T}_j)$, using the definition of the universal subscheme. It follows that the fibre of the locally-free sheaf $\mathcal{T}_J^{[n]}:=\bigoplus_{j\in J} \mathcal{T}^{[n]}_j$ at $y$ is  $H^0(Z_y, \Oo_{Z_y} \otimes \mathcal{T}_J)$.

To state the first result in this section, consider the quotient $A:=\overline{\Pi}/(b^*)$ of the preprojective algebra for the framed McKay quiver $Q$, where $b^*$ is the unique arrow in $Q$ with head at vertex $\infty$. For each vertex $i \in Q_0$, write $e_i \in A$ for the idempotent determined by the trivial path in $Q$ at vertex $i$. Define $\Bar{J} = \{\infty\} \cup J$ and $e_{\Bar{J}}:=e_{\infty} + \sum_{j\in J} e_j$, and consider the subalgebra 
\[
    A_{\Bar{J}}:=e_{\Bar{J}}Ae_{\Bar{J}}
\]
of $A$ spanned by the classes of paths with head and tail in $\Bar{J}$.
 
\begin{proposition}\label{prop:flatfamily}
    The locally-free sheaf $\mathcal{F}_n \coloneqq \mathcal{O}_{\Hilb^n(S_K)}\oplus \mathcal{T}_J^{[n]}$ on $\Hilb^n(S_K)$ is a flat family of $A_{\Bar{J}}$-modules of dimension vector $(1, n\delta\vert_J) \in \NN^{\Bar{J}}$. In particular, for each closed point $y\in \Hilb^n(S_K)$, the fibre of $\mathcal{F}_n$ at $y$ is an $A_{\Bar{J}}$-module of dimension vector $(1, n\delta\vert_J)$.
\end{proposition}
\begin{proof}
As explained in section \ref{sec:partial}, $\bigoplus_{j\in J} \mathcal{T}_j$ is a flat family of $\Pi_J$-modules that have dimension vector $\delta\vert_J$, so the maps $\mathcal{T}_{\tail(a)}\to \mathcal{T}_{\head(a)}$ for arrows $a$ in $Q_\Gamma$ satisfy the preprojective relations for the McKay quiver. Applying the functors $q^*$ and $p_*$ shows that the resulting maps $\mathcal{T}^{[n]}_{\tail(a)}\to \mathcal{T}^{[n]}_{\head(a)}$ also satisfy the preprojective relations, so $\mathcal{T}_J^{[n]}$ is a flat family of $\Pi_J$-modules that all have dimension vector $n\delta\vert_J$. Now, \cite[Lemma~3.2]{cggs_punctual-hilb_21} shows the algebra $A_{\Bar{J}}$ is obtained from $\Pi_J$ simply by adding generators given by the classes of the idempotent $e_\infty$ and the arrow $b\colon \infty\to 0$ (which plays no essential role in the relations of $A_{\Bar{J}}$). Therefore, we augment the flat family of $\Pi_J$-modules $\mathcal{T}_J^{[n]}$ by adding the line bundle $\mathcal{O}_{\Hilb^n(S_K)}$, so that each fibre has an additional copy of $\CC$ for the framing vertex, and we add the map of sheaves $p^\#\colon \Oo_{\Hilb^n(S_K)} \rightarrow \mathcal{T}_0^{[n]} \cong p_*(\Oo_{\mathcal{Z}_n})$ associated to the morphism $p$ (as in \cite[III.2]{hartshorne_ag_77}), so each fibre has a map $\CC\to \CC^n$ corresponding to the arrow $b$. It follows that the fibre $\CC \oplus H^0(Z_y,\mathcal{O}_{Z_y}\otimes \mathcal{T}_J)$ of the resulting locally-free sheaf $\mathcal{F}_n$ over $y$ is therefore an $A_{\Bar{J}}$-module of dimension vector $(1, n\delta\vert_J)$.
\end{proof}
For a closed point $y\in \Hilb^n(S_K)$ and for $Z:= \mathcal{Z}_y$, we write $\rho(Z)$ for the $A_{\Bar{J}}$-module of dimension vector $(1, n\delta\vert_J)$ obtained as the fibre of $\mathcal{F}_n$ over $y$. 
\begin{remark}
    The construction of the flat family on $\Hilb^n(S)$ by Kuznetsov~\cite[Theorem~43]{kuznetsov_quiver-hilb_07} is not quite correct as written. There, a $\overline{\Pi}$-module $\rho(Z)$ is defined for each length $n$ subscheme $Z\subset S$, where the $\CC$-vector space at the framing vertex is defined to be $\Gamma(S,\mathcal{O}_{S})$. However, $\Gamma(S,\mathcal{O}_{S})$ is isomorphic to $\CC[x,y]^\Gamma$, rather than $\CC$ as claimed in \emph{loc.~cit.}, so  $\dim_\infty\rho(Z)\neq 1$. We resolve this issue by adding $\mathcal{O}_{\Hilb^n(S_K)}$ to $\mathcal{T}_J^{[n]}$ in $\mathcal{F}_n$, thereby adding a copy of the field to the fibre $H^0(Z_y,\mathcal{O}_{Z_y}\otimes \mathcal{T}_J)$ over $y\in \Hilb^n(S_K)$, together with the map $p^\#\colon \CC \to H^0(Z_y,\mathcal{O}_{Z_y}\otimes \mathcal{T}_J)$, in a canonical way.
\end{remark}

We end this section with a short Lemma which expands on the equivalence $\Phi_J$ from Theorem \ref{thm:partial}.

\begin{lemma}
\label{lem:derivedequiv}
The tilting equivalence from Theorem~\ref{thm:partial} satisfies $\Phi_J (\Oo_{S_K}) = \Pi_Je_0$, and for any subscheme $Z\subset S_K$ of length $n$, we have
 $\Phi_J( \Oo_Z) = H^0(S_K,\mathcal{O}_Z\otimes \mathcal{T}_J)=e_J \rho(Z)$.
\end{lemma}
\begin{proof}
For the first statement, we compute that $\Psi_J(\Pi_Je_0) = \mathcal{T}_J^\vee \otimes \Pi_Je_0 = \mathcal{T}_0^\vee\cong \mathcal{O}_{S_K}$, and the result follows by applying $\Phi_J$.
For the second statement, we have 
\[
\Phi_J(\Oo_Z):= \mathbf{R}\!\Hom(\mathcal{T}_J^\vee, \Oo_Z) \cong \mathbf{R}\!\Hom(\Oo_{S_K}, \Oo_Z \otimes \mathcal{T}_J
) 
\cong H^0(S_K, \mathcal{O}_Z\otimes \mathcal{T}_J) 
\]
which is $e_J \rho(Z)$,  where $\rho(Z)$ is the fibre of  $\mathcal{F}_n$ over the point of $\Hilb^n(S_K)$ determined by $Z$.
\end{proof}

\subsection{A fine moduli space}

Following \cite[Proposition~3.3]{cggs_punctual-hilb_21}, there is a finite quiver $Q^*_{\Bar{J}}$ with vertex set $\Bar{J} = \{\infty\} \sqcup J$ and a $\CC$-algebra epimorphism $\beta_J \colon \CC Q^*_{\Bar{J}}\to A_{\Bar{J}}$. As a result, one can construct moduli spaces of $A_{\Bar{J}}$-modules. For this, let $\Rep(A_{\Bar{J}},(1, n\delta\vert_J))$ denote the affine scheme of representations of $Q^*_{\Bar{J}}$ that have dimension $(1, n\delta\vert_J)$ and satisfy the relations $\ker(\beta_J)$. The group $G(1, n\delta\vert_J):= \GL(1)\times \prod_{j\in J} \GL(n\dim\rho_j)$ acts on $\Rep(A_{\Bar{J}},(1, n\delta\vert_J))$. Let $h_J \coloneqq \sum_{j \in J} \delta_j \in \NN$, and define $\eta \colon \ZZ\oplus \ZZ^J\to \QQ$ satisfying $\eta((1, n\delta\vert_J))=0$ by setting
\begin{equation}
\label{eqn:specificeta}
    \eta(\rho_j) = \left\{\begin{array}{cr}
    1 &  \text{for }j \in J \setminus \{0\} \\
    \frac{1}{2n} - h_J +1 & \text{for }j=0, \end{array}\right.
\end{equation}
 so $\eta(\rho_\infty)=-\frac{1}{2}$. Then $\eta$ is a stability condition  for $A_{\Bar{J}}$-modules of dimension vector $(1, n\delta\vert_J)$.
\begin{lemma}
Every $\eta$-semistable $A_{\Bar{J}}$-module of dimension vector $(1, n\delta\vert_J)$ is $\eta$-stable.
\end{lemma}
\begin{proof}
 It suffices to show that every $v'\in \NN\oplus \NN^J$ satisfying $0<v'<(1, n\delta\vert_J)$ also satisfies $\eta(v') \neq 0$.
 Suppose otherwise, so there exists $v'\in \NN\oplus \NN^J$ satisfying $0<v'<(1, n\delta\vert_J)$ and $\eta(v') =0$. Substituting from \eqref{eqn:specificeta} gives
 \begin{equation}
     \label{eqn:etav'}
 0=\eta(v') = 
    -\frac{v'_\infty}{2} + v'_0\left(\frac{1}{2n} -h_J+1\right) + \sum_{j \in J\backslash\{0\}} v'_j .
    \end{equation}
 Since $0 < v' <(1, n\delta\vert_J)$, we must have $v'_\infty\in \{0,1\}$. We treat these cases separately:
 
 \smallskip
 \textsc{Case 1:} If $v'_\infty=0$, then since $v'_j\in \NN$ for all $j\in J$, the equality \eqref{eqn:etav'} forces $v'_0(\frac{1}{2n}) \in \NN$. We know that $v'_0\leq n$ because $v'< (1, n\delta\vert_J)$, so this can happen only if $v'_0=0$, which in turn implies that $\sum_{j \in J\setminus\{0\}} v'_j=0$. Therefore  $v'=0$, a contradiction. 

 \medskip
 
 \textsc{Case 2:} Otherwise, $v'_\infty=1$. Again, the fact that $v'_j\in \NN$ for all $j\in J$, together with equality \eqref{eqn:etav'} forces $-\frac{1}{2} +v'_0(\frac{1}{2n}) \in \NN$. Since  $v'_0\leq n$, this can only happen if $v'_0=n$. Substituting back into equation \eqref{eqn:etav'} shows that $\sum_{j \in J\setminus\{0\}} (n\delta_j-v'_j)=0$. This is a contradiction as $v' < (1, n\delta\vert_J)$.
 
 \smallskip
 
 \noindent This completes the proof.
  \end{proof}

It follows from King~\cite[Proposition~5.3]{King94} that the GIT quotient
 \[
 \mathcal{M}_{\eta}(A_{\Bar{J}},(1, n\delta\vert_J)) \coloneqq \Rep(A_{\Bar{J}}, (1, n\delta\vert_J)) \git_{\eta} G(1, n\delta\vert_J)
 \]
 is the fine moduli space of isomorphism classes of $\eta$-stable $A_{\Bar{J}}$-modules of dimension vector $(1, n\delta\vert_J)$. We choose the normalisation of the tautological locally-free sheaf on $\mathcal{M}_{\eta}(A_{\Bar{J}},(1, n\delta\vert_J))$ so that the rank-one summand indexed by the framing vertex $\infty$ is the trivial bundle $\mathcal{O}_{\mathcal{M}_{\eta}(A_{\Bar{J}},(1, n\delta\vert_J))}$.

  \begin{proposition}
  \label{prop:Kuznetsov}
    Consider the stability condition $\eta$ from \eqref{eqn:specificeta}. The $A_{\Bar{J}}$-modules of dimension vector $(1, n\delta\vert_J)$ in the flat family $\mathcal{F}_n$ on $\Hilb^n(S_K)$ are $\eta$-stable.
    In particular, the universal property of $\mathcal{M}_{\eta}(A_{\Bar{J}},(1, n\delta\vert_J))$ determines a universal morphism $\nu'\colon \Hilb^n(S_K)\to \mathcal{M}_{\eta}(A_{\Bar{J}},(1, n\delta\vert_J))$.
  \end{proposition}

 Our proof of this proposition follows closely the proof of the analogous statement for the minimal resolution $S$ by Kuznetsov~\cite[Proposition~44]{kuznetsov_quiver-hilb_07}. The key step is to study the possible dimension vectors of submodules of $\rho(Z)$. 

\begin{lemma}\label{lem:dimofsubrep}
 Let $\rho' \subsetneq \rho(Z)$ be a proper nonzero $A_{\Bar{J}}$-submodule. Then either:
    \begin{enumerate}
        \item[\one] 
        $\dim \rho' = (0,w)\in \NN\oplus \NN^J$, where $w$ satisfies $w_0\delta\vert_J \leq w \leq n\delta\vert_J$; or
        \item[\two] $\dim \rho' = (1,w)\in \NN\oplus \NN^J$, where $w$ satisfies 
        $w_0\delta\vert_J < w < n\delta\vert_J$.
    \end{enumerate}
\end{lemma}

\begin{proof}
    Since $\rho' \subset \rho(Z)$ is a proper $A_{\Bar{J}}$-submodule, the claims $w \leq n\delta\vert_J$ in case $\one$ and $w < n\delta\vert_J$ in case $\two$ are automatically fulfilled. We proceed by induction on the length of the subscheme.
    
    For the base case, let $Z:=\{x\}\subset S_K$ be a closed point. Then $e_J \rho(Z)\cong H^0(S_K, \Oo_{x} \otimes \mathcal{T}_J)$ is isomorphic to the fibre of the tautological bundle $\mathcal{T}_J$ on the fine moduli space $S_K\cong \mathcal{M}_{\chi_J}(\Pi_J,\delta\vert_J)$ from Proposition~\ref{prop:PartialMcKayquiverver}, so it is stable with respect to the stability parameter $\chi_J$ introduced in \eqref{eqn:chiJ}. If the $A_{\Bar{J}}$-submodule $\rho'$ satisfies $\dim \rho' = (0,w)$, then $\rho'=e_J \rho'\subseteq e_J \rho(Z)$ is an $\Pi_J$-submodule. Either $w_0=0$, in which case $w \geq w_0\delta\vert_J$, or $w_0=1$, in which case $\rho'=e_J\rho(Z)$ by $\chi_J$-stability, giving $w = w_0\delta\vert_J$. 

   Therefore, the result holds in case $\one$. Alternatively, $\dim \rho' = (1,w)$, but the map $p^\#$ from the proof of Proposition~\ref{prop:flatfamily} is injective by \cite[Proposition~5.4.3]{GrothendieckEGA1ver2:71}, so in particular, in this case  
$w_0 = 1$. Therefore, $e_J\rho'=e_J\rho(Z)$ by $\chi_J$-stability, but this would mean $\dim \rho' = (1, \delta\vert_J)=\dim \rho(Z)$, a contradiction, so case \two\ does not occur when $n=1$.

Assume the result for the $A_{\Bar{J}}$-module of dimension vector $(1, (n-1)\delta\vert_J)$ obtained as the fibre of $\mathcal{F}_{n-1}$ over $\Hilb^{n-1}(S_K)$, and let $Z_n\subset S_K$ be a subscheme of length $n$. Choosing a subscheme $Z_{n-1}\subset Z_n$ determines a closed point $x\in S_K$ and a short exact sequence of coherent sheaves
 \begin{equation}
 \begin{tikzcd}
 \label{eqn:OxOZ}
0 \arrow[r]& \Oo_{x} \arrow[r]& \Oo_{Z_{n}} \arrow[r] & \Oo_{Z_{n-1}} \arrow[r] & 0
\end{tikzcd}
\end{equation}
on $S_K$. Apply the exact equivalence $\Phi_J$ as in Lemma~\ref{lem:derivedequiv} to obtain a short exact sequence
\begin{equation}
    \label{eqn:alpha}
 \begin{tikzcd}
0 \arrow[r]& H^0(S_K,\Oo_{x}\otimes \mathcal{T}_J) \arrow[r]& H^0(S_K,\Oo_{Z_{n}}\otimes \mathcal{T}_J)  \arrow[r,"\phi"] & H^0(S_K,\Oo_{Z_{n-1}}\otimes \mathcal{T}_J)  \arrow[r] & 0
\end{tikzcd}
\end{equation}
of $\Pi_J$-modules. Note that $\ker(\phi)$ is isomorphic to the fibre of $\mathcal{T}_J$ at $x$, so it's $\chi_J$-stable. Add the identity map $\text{id}\colon \CC\to \CC$ to $\phi$ for vertex $\infty$, and include maps $p^\#\colon \CC\to H^0(S_K,\mathcal{O}_{Z_n}\otimes \mathcal{T}_0)$ and $\phi_0\circ p^\#\colon \CC\to H^0(S_K,\mathcal{O}_{Z_{n-1}}\otimes \mathcal{T}_0)$ for the arrow with tail at $\infty$, where $\phi_0$ is the graded piece of the map $\phi$ corresponding to vertex $0\in J$. The result is a short exact sequence
 \begin{equation}
     \label{eqn:f}
 \begin{tikzcd}
0 \arrow[r]& H^0(S_K,\Oo_{x}\otimes \mathcal{T}_J) \arrow[r]& \rho(Z_n) \arrow[r,"\psi"] & \rho(Z_{n-1}) \arrow[r] & 0
\end{tikzcd}
\end{equation}
of $A_{\Bar{J}}$-modules; here, we regard $ H^0(S_K,\Oo_{x}\otimes \mathcal{T}_J)$ as an $A_{\Bar{J}}$-module of dimension vector $(0,\delta\vert_J)$.  

 For case \one, we assume that $\dim\rho' = (0,w)$. Then $\rho'=e_J \rho'\subseteq H^0(S_K,\mathcal{O}_{Z_n}\otimes \mathcal{T}_J)$ is an $\Pi_J$-submodule. Let $\phi'$ denote the restriction of $\phi$ to  $\rho'$, write $u:=\dim \ker(\phi')\in \NN^J$ and $w':=\dim \operatorname{im}(\phi')\in \NN^J$. The base case gives $u\geq u_0\delta\vert_J$ and the inductive hypothesis gives $w'\geq w'_0\delta\vert_J$, so $w=u+w'\geq u_0\delta\vert_J + w'_0\delta\vert_J = w_0\delta\vert_J$ by additivity in \eqref{eqn:alpha}. Therefore the result holds in case \one. 

For case \two, assume that $\dim \rho'=(1,w)$. Let $\psi'$ denote the restriction of $\psi$ to  $\rho'$, let $u:=\dim \ker(\psi')\in \NN^J$ and let $\dim \im(\psi')=(1,w')$, where $w' \in \NN^J$, giving $w=u+w'$. 
 If $u>u_0\delta\vert_J$ or $w' > w'_0\delta\vert_J$ then we are done by additivity, so suppose otherwise. To obtain a contradiction, we claim first that the sequences \eqref{eqn:OxOZ}--\eqref{eqn:f} all split in this case. Indeed, the induction assumption allows us to assume $w'=(n-1)\delta\vert_J$, so $\psi'(\rho') = \rho(Z_{n-1})$. We have $u = u_0\delta\vert_J$ and $u_0\leq 1$. If $u_0=1$, we have $w=n\delta\vert_J$, but this is absurd since $\rho'$ is a proper submodule, so $u_0=0$ and hence $u=0$. It follows that $\rho' \cong\rho(Z_{n-1})$, and in particular, $\rho(Z_{n-1})$ is an $A_{\Bar{J}}$-submodule of $\rho(Z_n)$. That is, sequence \eqref{eqn:f} splits. Apply $e_J$ on the left in \eqref{eqn:f} to see that the sequence of $\Pi_J$-modules from \eqref{eqn:alpha} splits. The additional data in \eqref{eqn:f} comprising the maps from the framing vertex in $\rho(Z_n)$ and $\rho(Z_{n-1})$ is equivalent to the $\Pi_J$-module homomorphisms from $\Pi_Je_0$ given by the image of the canonical projections $\Oo_S \to \Oo_{Z_{n}}$ and $\Oo_S \to \Oo_{Z_{n-1}}$ under the derived equivalence $\Phi_J$. We record this in the following commutative diagram of $\Pi_J$-modules
\[
\begin{tikzcd}
 & & \Pi_Je_0 \arrow[r, leftrightarrow] \arrow[d]& \Pi_Je_0 \arrow[d]\\
0 \arrow[r]& H^0(S_K,\Oo_{x}\otimes \mathcal{T}_J) \arrow[r]& H^0(S_K,\Oo_{Z_{n}}\otimes \mathcal{T}_J)  \arrow[r,"\phi"] & H^0(S_K,\Oo_{Z_{n-1}}\otimes \mathcal{T}_J)  \arrow[r] & 0,
\end{tikzcd}
\]
where the bottom row splits. The fact that $\rho' \cong \rho(Z_{n-1})$ is a subrepresentation of $\rho(Z_n)$ means that the image of the middle vertical homomorphism must stay within $e_J\rho' \cong H^0(S_K,\Oo_{Z_{n-1}}\otimes \mathcal{T}_J)$.

Apply the exact equivalence $\Psi_J$ from Theorem~\ref{thm:partial}\ to the diagram and, by Lemma~\ref{lem:derivedequiv}, we obtain a commutative diagram
\begin{equation}
\begin{tikzcd}
\label{eqn:snake}
 &  0\arrow[d]\arrow[r]  & \Oo_{S_K} \arrow[r, leftrightarrow] \arrow[d] & \Oo_{S_K} \arrow[d]\arrow[r] & 0\\
  0 \arrow[r]&   \Oo_{x} \arrow[r] & \Oo_{Z_n} \arrow[r] & \Oo_{Z_{n-1}}\arrow[r] &  0
  \end{tikzcd}
\end{equation}
of sheaves with exact rows, where the bottom row splits as claimed. However, then the middle vertical homomorphism must factor through the summand $\Reg_{Z_{n-1}}$, which contradicts the assumption that it is surjective. \qedhere

\end{proof}

\begin{proof}[Proof of Proposition~\ref{prop:Kuznetsov}]
 For the stability condition $\eta$ from \eqref{eqn:specificeta} and for any subscheme $Z\subset S_K$ of length $n$, we must show that $\eta(\rho^\prime)>0$ for every nonzero proper $A_{\Bar{J}}$-submodule $\rho'\subset\rho(Z)$. Lemma~\ref{lem:dimofsubrep} shows there are two cases:

\medskip

\noindent \textsc{Case 1:} $\dim \rho^\prime =(0,w)$ for some nonzero $w=(w_j)\in \NN^J$ satisfying $w_0\delta\vert_J \leq w\leq n\delta\vert_J$.

\smallskip
We substitute the values of $\eta(\rho_j)$ for $j\in J$ from \eqref{eqn:specificeta} to see that
\[
\eta(\rho')= \left(\frac{1}{2n} - h_J +1\right)w_0 + \sum_{j \in J\setminus\{0\}} w_j = 
\frac{w_0}{2n} + \sum_{j \in J\setminus\{0\}} \big(w_j - w_0 \delta_j\big).
\] 
Since $w \geq w_0\delta\vert_J$, we have $w_j \geq w_0\delta_j$ for all $j \in J\setminus \{0\}$. Therefore, $\eta(\rho') \geq 0$, with equality if and only if $w_0=0$ and $w_j = w_0\delta_j =0$ for all $j \in J\setminus\{0\}$. But $w\neq 0$ as $\rho'\neq 0$, so $\eta(\rho')>0$.

\medskip

\noindent \textsc{Case 2:} $\dim \rho^\prime =(1,w)$ for some $w=(w_j)\in \NN^J$ satisfying $w> w_0\delta\vert_J$. 

\smallskip
Again, substitute in the values of $\eta(\rho_i)$ for $i\in \{\infty\}\cup J$ from \eqref{eqn:specificeta} to get
\[
    \eta(\rho') = -\frac{1}{2} + \left(\frac{1}{2n} - h_J +1\right) w_0 + \sum_{j \in J\backslash\{0\}} w_j = -\frac{1}{2} + \frac{w_0}{2n} + \sum_{j \in J\setminus\{0\}} (w_j-w_0\delta_j\big).
\]
Since $w>w_0\delta\vert_J$, we have $w_j \geq w_0\delta_j$ for all $j \in J \setminus\{0\}$, and there exists $l$ with $w_l > w_0\dim(\rho_l)$. Therefore, $\sum_{j \in J\setminus\{0\}} (w_j-w_0\delta_j)$ is a (strictly) positive integer, so $\eta(\rho')>0$ as required.
\end{proof}

\subsection{Geometric construction of the isomorphism}

We now show the universal morphism $\nu'$ does define an isomorphism from $\Hilb^n(S_K)$ to $\NQV_{\theta'_K}(n\delta,\Lambda_0 )$. First we explain the use of $\mathcal{M}_\eta(A_{\Bar{J}},(1, n\delta\vert_J))$. 

\begin{lemma}
\label{lem:CYisomorphism}
For $\eta$ from \eqref{eqn:specificeta} and $\theta'_K\in \relint(\sigma_K)$, there is a commutative diagram
  \begin{equation}
      \label{eqn:CYmorphism}
\begin{tikzcd}
\NQV_{\theta'_K}(n\delta,\Lambda_0 ) \arrow[r,hook, "\varphi"]\arrow[d, "\lambda'"] & \mathcal{M}_\eta(A_{\Bar{J}},(1, n\delta\vert_J))\arrow[d, "\lambda"] \\
\NQV_0(n\delta, \Lambda_0 )  \arrow[r, hook, "\varphi_0"] & \mathcal{M}_0(A_{\Bar{J}},(1, n\delta\vert_J)),
\end{tikzcd}
\end{equation}
 where the left hand map is a crepant partial resolution and both horizontal maps are closed immersions, where both $\varphi$ and $\varphi_0$ send each semistable $\overline{\Pi}$-module $B$ to $e_{\overline{J}}B$.
\end{lemma}
\begin{proof}
Since quiver varieties in the same GIT cone are isomorphic, we consider $\overline{\eta}$ defined by setting
\begin{equation}
\label{eqn:overlineeta}
    \overline{\eta}(\rho_i) = \left\{\begin{array}{cr}
    1 &  \text{for }i \in J \setminus \{0\} \\
    \frac{1}{2n} - h_J +1 & \text{for }i=0 \\
    0 & \text{for } i \in K \end{array}\right.,
\end{equation}
where, as before, $h_J \coloneqq \sum_{j \in J} \delta_j
\in \NN$. It follows that  $\overline{\eta}(\rho_\infty)=-\frac{1}{2}$. Since $\overline{\eta}$ lies in the relative interior of the GIT cone $\sigma_K$, there is an isomorphism $\NQV_{\theta'_K}(n\delta,\Lambda_0 )\cong \mathfrak{M}_{\overline{\eta}}(n\delta, \Lambda_0 )$ of schemes over $\NQV_{0}(n\delta,\Lambda_0 )$. Craw-Yamagishi \cite{craw_yamagishi_hilb-can_26} construct the commutative diagram \eqref{eqn:CYmorphism} in which the vertical maps are projective morphisms obtained by VGIT and the horizontal maps are closed immersions. The definition of $\varphi$ ensures that $\varphi(B)=e_{\overline{J}}B$.
\end{proof}
We now see that the universal morphism $\nu$ fits into a commutative diagram with $\Sym^n(\Sing)$.
\begin{lemma}\label{lem:comm-diagram-cornered}
    There is a commutative diagram
    \begin{equation}
    \label{eqn:kappalambda}
        \begin{tikzcd}
    \Hilb^n(S_K) \arrow[r, "\nu'"] \arrow[d, swap, "\kappa"]& \mathcal{M}_\eta(A_{\Bar{J}},(1, n\delta\vert_J)) \arrow[d, "\lambda"]\\
    \Sym^n(\Sing) \arrow[r, "\nu_0'"] & \mathcal{M}_0(A_{\Bar{J}},(1, n\delta\vert_J))
\end{tikzcd}
\end{equation}
where both vertical morphisms are projective, the lower horizontal map is a closed immersion and the upper horizontal map is injective on closed points.
    
\end{lemma}
\begin{proof}
    The morphism $\kappa$ is the composition of the Hilbert Chow morphism $\Hilb^n(S_K)\rightarrow \Sym^n(S_K)$ and the functorial morphism $\Sym^n(S_K) \rightarrow \Sym^n(\Sing)$. Define $\nu_0':=\varphi_0\circ \nu_0$, where $\varphi_0$ is the closed immersion from \eqref{eqn:CYmorphism} and the isomorphism 
 \[
 \nu_0\colon \Sym^n(\Sing)\longrightarrow \Spec(\CC)\times \Sym^n\big(\mathcal{M}_0(\Pi, \delta)\big)\cong \NQV_{0}(n\delta,\Lambda_0 )
 \]
 is from \cite[Lemma~4.5]{bellamy_craw_birational-symplectic_20}. We already know that $\kappa$ is a projective crepant partial resolution, and $\lambda$ is obtained by VGIT, so it's projective. Next we show that the diagram commutes.
 
 Let $y \in \Hilb^n(S_K)$. Write the support of the associated subscheme $Z_y$ as $x_1, \dots, x_n$ in $S_K$ (not necessarily distinct points). By definition, $\kappa(y)=\sum_{1\leq i\leq n} g(x_i)$, where $g: S_K \rightarrow \Sing$. For each $x_i \in S_K$, we choose some $s_i \in h^{-1}(x_i) \subset f^{-1}(g(x_i))\subseteq S$. The corresponding point in $\mathcal{M}_\zeta(\Pi,\delta)$ is defined by $H^0(S, \mathcal{O}_{s_i}\otimes \mathcal{V})$, the fibre of $\mathcal{V}$ at $s_i$. Applying Remark 2.4 of \cite{craw_wye_hilb-crepant-partial_25-v2} shows that $\nu_0(\kappa(y))$ is the $\Seshadri$-equivalence class of the $\Pi$-module associated to the points $s_i$, that is: 
\[
 \nu_0(\kappa(y)) = \CC\oplus \bigoplus_{1\leq i\leq n} H^0(S,\mathcal{O}_{s_i}\otimes \mathcal{V}).
 \]
 Since $\varphi_0(B)=e_JB$ for any $0$-semistable $\Pi$-module $B$, we compute
\begin{equation}
\label{eqn:nu0'kappa}
  \nu_0'(\kappa(y)) := \varphi_0\big(\nu_0(\kappa(y)\big) = \CC\oplus \bigoplus_{1\leq i\leq n} H^0(S,\mathcal{O}_{s_i}\otimes \mathcal{V}_J).
\end{equation}

On the other hand, by definition $\nu'(y)=\CC\oplus H^0(S_K,\mathcal{O}_{Z_y}\otimes \mathcal{T}_J)$. The morphism $\lambda$ is obtained by VGIT, so $\lambda$ sends $\nu'(y)$ to the corresponding $0$-polystable $A_{\Bar{J}}$-module. The short exact sequence \eqref{eqn:alpha} appearing in our proof by induction shows that the objects appearing in the filtration of $\nu'(y)$ are $H^0(S_K,\mathcal{O}_{x_i}\otimes \mathcal{T}_J)$ for each $x_i$ in the support of $Z_y$ (and one copy of $\CC$), so the $S$-equivalence class of the $0$-semistable $A_{\Bar{J}}$-module $\nu'(y)$ is
\begin{equation}
\label{eqn:lambdanu}
    \lambda(\nu(y)) = \CC\oplus \bigoplus_{1\leq i\leq n} H^0(S_K,\mathcal{O}_{x_i}\otimes \mathcal{T}_J).
\end{equation}
To see that \eqref{eqn:nu0'kappa} and \eqref{eqn:lambdanu} agree, note first that $H^0(S, \Oo_{s_i} \otimes \mathcal{V}_j) \cong H^0(S_K, h_*(\Oo_{s_i} \otimes \mathcal{V}_j))$. However, $\mathcal{V}_j \cong h^*(\mathcal{T}_j)$ by Lemma \ref{lem:push-taut-prop}(\ref{lem:push-taut-prop:3}), so
\[
    H^0(S, \Oo_{s_i} \otimes \mathcal{V}_j) \cong H^0(S_K, h_*(\Oo_{s_i} \otimes h^*(\mathcal{T}_j)))\cong H^0(S_K, h_*(\Oo_{s_i}) \otimes \mathcal{T}_j) \cong H^0(S_K,\Oo_{x_i}\otimes \mathcal{T}_j)
\]
by the projection formula. Take the direct sum over $1\leq i\leq n$ and $j\in J$ to see that \eqref{eqn:nu0'kappa} and \eqref{eqn:lambdanu} agree, so diagram \eqref{eqn:kappalambda} commutes. 

We now show that $\nu'$ is injective on closed points. The $A_{\Bar{J}}$-module $\nu'(y)$ is equivalent to the data of the homomorphism of $\Pi_J$-modules \[[\Pi_J e_0 \to H^0(\Reg_{Z_y} \otimes \mathcal{T}_J)] = \Hom(\mathcal{T}_J^{\vee}, [\Reg_{S_K} \twoheadrightarrow \Reg_{Z_y}]) \; .\] By Lemma \ref{lem:derivedequiv}, this this is the same as $\Phi_J([\Reg_{S_K} \twoheadrightarrow \Reg_{Z_y}])$. But $\Phi_J$ is an equivalence by Theorem \ref{thm:partial}(\ref{thm:partial:1}), so $\nu'(y)$ determines the subscheme $Z_y$ uniquely. \qedhere
\end{proof}

We now identify $\NQV_{\theta'_K}(n\delta, \Lambda_0 )$ and $\NQV_{0}(n\delta, \Lambda_0 )$ with closed subschemes of $\mathcal{M}_\eta(A_{\Bar{J}},(1, n\delta\vert_J))$ and $\mathcal{M}_0(A_{\Bar{J}},(1, n\delta\vert_J))$ respectively determined by the closed immersions from Lemma~\ref{lem:CYisomorphism}. By definition, the morphism $\nu_0'$ from Lemma~\ref{eqn:kappalambda} factors via the closed subscheme $\NQV_{0}(n\delta, \Lambda_0 )$ of $\mathcal{M}_0(A_{\Bar{J}},(1, n\delta\vert_J))$.

\begin{theorem}
\label{thm:isomophHilbtoNQV}
    Restricting the codomains of $\nu'$ and $\nu_0'$ to their images gives a commutative diagram 
    \begin{equation}
    \label{eqn:Hilbtomathfrak}
        \begin{tikzcd}
            \Hilb^n(S_K) \arrow[r, "\nu_J"] \arrow[d, "\kappa"]&\NQV_{\theta'_K}(n\delta, \Lambda_0 )\arrow[d, "\lambda'"]\\
            \Sym^n(\Sing) \arrow[r, "\nu_0"]&\NQV_{0}(n\delta, \Lambda_0 )
        \end{tikzcd}
    \end{equation}
    where the horizontal maps are isomorphisms and $\lambda'$ is the restriction of $\lambda$ to $\NQV_{\theta'_K}(n\delta, \Lambda_0 )$.
\end{theorem}
\begin{proof}
   Since $\nu_0'=\varphi_0\circ \nu_0$ is the composition of an isomorphism and a closed immersion, it follows that restricting the codomain of $\nu_0'$ to the image is the isomorphism $\nu_0$. Since the diagram \eqref{eqn:kappalambda} commutes, the image of $\nu'$ is contained in the closed subscheme $\lambda^{-1}(\NQV_{0}(n\delta, \Lambda_0 ))$ of $\mathcal{M}_\eta(A_{\Bar{J}},(1, n\delta\vert_J))$. We know $\nu'$ is projective \cite[II.\ Ex~4.8,4.9]{hartshorne_ag_77}, because $\nu_0 \circ \kappa$ is projective, and $\lambda'$ is projective and hence separated. Therefore, the image of $\nu'$ is closed. We next notice that the morphisms in diagrams \eqref{eqn:kappalambda} and \eqref{eqn:CYmorphism} are isomorphisms when restricted to the open dense locus in $\Sym^n(\Sing)$ consisting of cycles supported at n distinct non-singular points. Since $\Hilb^n(S_K)$ is irreducible and proper over the base $\Sym^n(\Sing)$, the image of $\Hilb^n(S_K)$ under $\nu'$ is precisely the closure of the image of this open dense locus.
   But so is the closed, irreducible subset $\NQV_{\theta'_K}(n\delta, \Lambda_0 ) \subseteq \mathcal{M}_\eta(A_{\Bar{J}},(1, n\delta\vert_J))$, so we must have that the image of $\Hilb^n(S_K)$ under $\nu'$ is precisely $\NQV_{\theta'_K}(n\delta, \Lambda_0 )$.
   Since $\Hilb^n(S_K)$ is reduced, this implies that $\nu'$ factors through $\nu_J$ as required.
   
   It remains to show that $\nu_J$ is an isomorphism. Note that $\nu_J$ is injective on closed points because so is $\nu'$. Since $\nu_J$ is also projective, birational, and the target is normal, then it's an isomorphism by Zariski's Main Theorem. \qedhere
\end{proof}

\section{The isomorphism between the smooth moduli spaces}\label{sec:proofsmooth}

Recall that $\Quot_J^{n \delta}$ is a fine moduli scheme for quotients of sheaves \[\phi \colon h_*\mathcal{V} \twoheadrightarrow F\] on $S_K$. For any such quotient $\phi$, consider the homomorphism of $\Pi$-modules \[H^0(\phi) \colon \Pi e_0 = H^0(\mathcal{V}) \to H^0(F)\] Note that the homomorphism $H^0(\phi)$ is determined by $H^0(\phi)(e_0) \in H^0(F)$. From this data, we define a $\Bar{\Pi}$-module structure on \[\C \oplus H^0(F)\] as follows. We let the left-hand summand $\C$ to be the component at $\infty$, let the arrow $\infty \to 0$ send $1 \in \C$ to $H^0(\phi)(e_0)$, and let $0 \to \infty$ act by $0$. The preprojective relations defining $\Bar{\Pi}$ reduce to those defining $\Pi$ when the arrow $0 \to \infty$ is $0$, and since $H^0(F)$ is a $\Pi$-module, the structure on $\C \oplus H^0(F)$ indeed satisfies the preprojective relations. We will slightly abuse notation and denote this $\Bar{\Pi}$-module also by $H^0(\phi)$.

\begin{lemma}\label{lem:zero-generates}
    Let $\phi \colon h_*\mathcal{V} \twoheadrightarrow F$ be a surjection of $h_* \Ealg$-modules with $F$ of finite length. Then there is no proper $\Pi$-submodule of $H^0(F)$ containing all of $e_0 H^0(F)$. 
\end{lemma}
\begin{proof}
    The module $H^0(F)$ splits into a direct sum indexed by the points supporting $F$, so it suffices to prove the statement assuming that $F$ is supported at a closed point $x \in S_K$. Let $i \in I$ and let $s \in e_i H^0(F)$. Since $\phi$ is surjective, there exists a section $u \in e_i h_* \mathcal{V}(U)$ with $\phi(u) = s|_U$ over some open subset $U \subseteq S_K$ containing $x$.

    The sheaf $e_i h_* \mathcal{V}$ is globally generated by Lemma \ref{lem:push-taut-prop}(\ref{lem:push-taut-prop:1}), so the natural homomorphism $H^0(e_i h_* \mathcal{V}) \otimes \Reg_{S_K} \to e_i h_* \mathcal{V}$ is surjective. This means that, possibly after further shrinking the neighbourhood $U$ of $x$, there exist global sections $u_1, \ldots, u_l \in H^0(e_i h_* \mathcal{V})$ and local sections $t_1, \ldots, t_l \in \Reg_{S_K}(U)$ such that \begin{equation}\label{eq:zero-generates} u = \sum_{j = 1}^l u_j t_j \; .\end{equation} Under the identification $e_i h_* \mathcal{V} = e_i h_* \Ealg e_0$ and $\Reg_{S_K} \cong e_0 h_* \mathcal{V}$, we may view the $u_j$ as global sections of $e_i h_* \Ealg e_0$ and the $t_j$ as local sections of $e_0 h_* \mathcal{V}$, and the multiplication in (\ref{eq:zero-generates}) as the action of $h_* \Ealg$ on $h_* \mathcal{V}$. This gives \[s|_U = \phi(u) = \phi\left(\sum_{j = 1}^l u_j t_j\right) = \sum_{j = 1}^l u_j \phi(t_j) \; ,\] where now $\phi(t_j) \in e_0 F(U) = e_0 H^0(F)$ and $u_j \in e_i \Pi e_0$ for all $j$. This shows that $e_0 H^0(F)$ generates $H^0(F)$ as a $\Pi$-module. 
\end{proof}

Recall that the morphism $\Quot_J^{n \delta} \to \Hilb^n(S_K)$ (\ref{eq:quot-to-hilb}) is given by sending $\phi$ to $e_0 \phi \colon \Reg_{S_K} \to e_0 F$. Cornering is exact, so $e_0 \phi$ is surjective, and $e_0 F = \Reg_Z$ for some length-$n$ subscheme $Z$ of $S_K$.

On the other hand, we can corner at the subset $J$ and obtain a homomorphism $e_J \phi \colon \mathcal{T}_J = h_* \mathcal{V}_J \to e_J F$ of sheaves of $h_* \Ealg_J$-modules. Since $\mathcal{T}_J$ is locally-free, further cornering at $e_0$ defines a Morita equivalence between $h_* \Ealg_J = \SEnd(\mathcal{T}_J)$ and $\Reg_{S_K}$, so we can reconstruct $e_J \phi$ from $e_0 \phi$. In fact, in the same way as above, the data of $H^0(e_J \phi) \colon e_J \Pi e_0 \to H^0(e_J F)$ is encoded in the $\Bar{\Pi}_{\Bar{J}}$-module $e_{\Bar{J}} H^0(\phi) = \C \oplus H^0(e_J F)$, and this module is the same as $\rho(Z)$ from section \ref{sec:geometric}. Here, $\Bar{J} = \{\infty\} \cup J$, $e_{\Bar{J}}$ is the associated idempotent, and $\Bar{\Pi}_{\Bar{J}} = e_{\Bar{J}} \Bar{\Pi} e_{\Bar{J}}$.

\begin{proposition}\label{prop:mor-to-nqv}
    \ \begin{enumerate}
        \item\label{prop:mor-to-nqv:1} For every surjection $\phi \colon h_* \mathcal{V} \twoheadrightarrow F$ over $h_* \Ealg$, where $F$ is of finite length and dimension vector $n \delta$, the $\Bar{\Pi}$-module $H^0(\phi)$ is $\theta_K$-stable for $\theta_K \in C_K$.

        \item\label{prop:mor-to-nqv:2} There is a morphism \[\Quot_J^{n \delta} \longrightarrow \NQV_{\theta_K}(n \delta,\Lambda_0 )\] which gives the mapping $\phi \mapsto H^0( \phi)$ on closed points.
        
        \item\label{prop:mor-to-nqv:4} The morphism from (\ref{prop:mor-to-nqv:2}) is injective on closed points.
    \end{enumerate}
\end{proposition}
\begin{proof}
    \begin{enumerate}
        \item As explained above, the surjection $\phi$ induces a surjection \[e_J \phi \colon e_J (h_* \mathcal{V}) \twoheadrightarrow e_J F \] over  $h_*\Ealg_J$ which is equivalent to defining a length-$n$ subscheme $Z$ of $S_K$. The corresponding $\Bar{\Pi}_{\Bar{J}}$- algebra $\rho(Z)$ is the same as $e_{\Bar{J}} H^0(\phi)$.

        To show that $\theta_K$-stability holds for any $\theta_K \in C_K$, it suffices to show it for one specific choice of such $\theta_K$. We define $\theta_K \in C_K$ by setting \[\theta_K(\alpha_j) = nh \quad \text{for} \quad j \in J \setminus \{0\} \; , \quad \theta_K(\alpha_k) = 1 \quad \text{for} \quad k \in K \; , \quad \text{and} \quad \theta_K(\delta) = h \;, \] where $h = \sum_{i \in I} \delta_i$. (Note that this choice of $\theta_K$ differs from that defined in \cite[Proof of Lemma 3.10]{craw_wye_hilb-crepant-partial_25-v2}).
        
        Now, let $\Bar{M} \subset H^0(\phi)$ be a non-zero, proper submodule of dimension vector $(r,v) \in \{0,1\} \times \N^I$. Then $\Bar{N} \coloneqq e_{\Bar{J}} \Bar{M}$ is a submodule of $\rho(Z)$, proper by Lemma \ref{lem:zero-generates}. By Lemma \ref{lem:dimofsubrep}, its dimension vector \[\dim(\Bar{N}) = (r,v|_J) \in \{0,1\} \times \N^J\] satisfies:
        \begin{itemize}
            \item If $r=0$, then $v_0 \delta\vert_J \leq v|_J \leq n \delta\vert_J$;
            \item if $r=1$, then $v_0 \delta\vert_J < v|_J < n \delta\vert_J$.
        \end{itemize}
        
        Hence, if $r = 0$, \[\theta_K(\dim(\Bar{M})) = v_0 \theta_K(\delta\vert_J) + \theta_K(v - v_0\delta\vert_J) > 0\] because $\theta_K(\delta\vert_J) > 0$, $v - v_0 \delta\vert_J$ has non-negative entries and is supported on $I \setminus \{0\}$, and $v \neq 0$.
        
        If $r = 1$ on the other hand, \[\theta_K(\dim(\Bar{M})) = \theta_K(0,v-n\delta) = \theta_K(v - v_0 \delta\vert_J) + \theta_K(v_0 \delta\vert_J - n \delta) > 0\] because $\theta_K(n \delta - v_0 \delta\vert_J) \leq \theta_K(n \delta) = nh$ and $\theta_K(v-v_0 \delta\vert_J) \geq \theta_K(\alpha_j) = nh$ for a $j \in J \setminus \{0\}$ where $v_j > v_0 \delta_j$, and one of these two inequalities is strict because $v \neq 0$.

        \item Since $\NQV_{\theta_K}(n \delta,\Lambda_0 )$ is a fine moduli space for $\theta_K$-stable modules, all we need to show is that the $H^0(\phi)$ from (\ref{prop:mor-to-nqv:1}) fit into a flat family over $\Quot_J^{n \delta}$. For the remainder of this proof we shorten $\Quot_J^{n \delta}$ to $\Quot$. Consider the diagram \[\begin{tikzcd}\Quot \times S_K \arrow[r,"q"] \arrow[d,"p"] & S_K \\ \Quot & \end{tikzcd}\] together with the universal quotient \[\phi^{\univ} \colon q^*h_* \mathcal{V} \twoheadrightarrow \Fb \] of $q^*h_*\Ealg$-modules, flat over $\Quot$. Applying $p_*$ gives a homomorphism \[(\Pi e_0)_{\Quot} \cong p_* q^* h_* \mathcal{V} \to p_* \Fb \] of $\Pi_{\Quot}$-modules. We define a $\Bar{\Pi}_{\Quot}$-algebra structure on $\Reg_{\Quot} \oplus p_* \Fb$ in the same way as we did pointwise, which gives us the required flat family. Indeed, since $p_* \Fb$ is locally-free over $\Quot$, for every $y \in \Quot$, the natural homomorphism of $\Bar{\Pi}$-modules $(\Reg_{\Quot} \oplus p_* \Fb)_y \cong H^0(\phi^{\univ}_y)$ is an isomorphism.
        
        \item This part is analogous to the proof of injectivity on closed points in Lemma \ref{lem:comm-diagram-cornered}. The complexes $Rh_*(\mathcal{V}) = Rh_*(\Ealg)e_0$ and $Rf_*(\mathcal{V}) = Rf_*(\Ealg)e_0$ are concentrated in degree zero by Corollary \ref{cor:vanishing}. Hence, \[Rg_*(h_* \mathcal{V}) = Rg_* Rh_* \mathcal{V} = Rf_* \mathcal{V} = f_* \mathcal{V} = g_*(h_* \mathcal{V}) \; ,\] i.e.\ $Rg_*(h_* \mathcal{V})$ is supported in degree $0$. The same is true for $Rg_* F$ because $F$ has zero-dimensional support. Hence, by Theorem \ref{thm:derived-equivalence}, we can reconstruct $\phi \colon h_* \mathcal{V} \twoheadrightarrow F$ uniquely from the morphism $g_* \phi \colon g_* h_* \mathcal{V} \to g_* F$. Since $\Sing$ is affine, the latter morphism is equivalent to the morphism $H^0(\phi) \colon \Pi e_0 \to H^0(F)$, or equivalently, the $\Bar{\Pi}$-module $H^0(\phi)$.\qedhere
    \end{enumerate}
\end{proof}

\begin{proof}[Proof of Theorem \ref{thm:isom-quot-nqv}]
    By combining (\ref{eq:quot-to-hilb}), Theorem \ref{thm:isomophHilbtoNQV}, and Proposition \ref{prop:mor-to-nqv}, we obtain the following diagram.
    \[\begin{tikzcd} \Quot_J^{n \delta} \arrow[r, "\nu"] \arrow[d, "e_0"] & \NQV_{\theta_K}(n \delta,\Lambda_0) \arrow[d] \\ \Hilb^n(S_K) \arrow[r, "\nu_J", "\sim"'] & \NQV_{\theta'_K}(n\delta, \Lambda_0 )\end{tikzcd}\]
    
    First, we show that the diagram commutes and that every morphism is projective and birational. We already know that all spaces involved are quasi-projective, $\nu_J$ is an isomorphism by Theorem \ref{thm:isomophHilbtoNQV}, $e_0$ is birational and proper by Proposition \ref{prop:quot-hilb-stack} and the VGIT morphism is birational and projective by \cite[Proposition 4.8]{bellamy_craw_birational-symplectic_20}. To see that the square commutes, it is enough to see that the two morphisms $\Quot_J^{n \delta} \to \NQV_{\theta'_K}(n\delta, \Lambda_0)$ agree after composition with the closed embedding $\NQV_{\theta'_K}(n\delta, \Lambda_0) \hookrightarrow \mathcal{M}_{\eta}(A_{\Bar{J}},(1, n\delta\vert_J))$. A point representing a quotient $\phi$ in $\Quot_J^{n \delta}$ is being mapped to the point in $\mathcal{M}_{\eta}(A_{\Bar{J}},(1, n\delta\vert_J))$ representing the cornered module associated to $\Hom_{S_K}(\mathcal{T}_J^{\vee}, e_0 \phi)$ by $\nu_J \circ e_0$, respectively $e_J H^0(\phi)$ along the other path around the square. These two modules are isomorphic, as
    \[\begin{split}
        \Hom_{S_K}(\mathcal{T}_J^{\vee}, e_0 \phi) & \cong H^0(\mathcal{T}_J^{\vee\!\vee} \otimes_{\Reg_{S_K}} e_0 \phi) \\
        & \cong H^0(\mathcal{T}_J \otimes_{\Reg_{S_K}} e_0 \phi) \\
        & \cong H^0(\mathcal{T}_J \otimes_{\Reg_{S_K}} e_0 h_*\!\Ealg \otimes_{h_*\!\Ealg} \phi) \\
        & \cong H^0(\mathcal{T}_J \otimes_{\Reg_{S_K}} \SHom_{S_K}(h_*\mathcal{V}, \Reg_{S_K}) \otimes_{h_*\!\Ealg} \phi) \\
        & \cong H^0(\SHom_{S_K}(h_*\mathcal{V}, \mathcal{T}_J) \otimes_{h_*\!\Ealg} \phi) \\
        & \cong H^0(e_J h_*\!\Ealg \otimes_{h_*\!\Ealg} \phi) \\
        & \cong H^0(e_J \phi) \cong e_J H^0(\phi) \; ,
    \end{split}\] where we used that $\mathcal{T}_J$ is locally-free in the first, second and fifth isomorphism. Since $\Quot_J^{n \delta}$ is reduced, this pointwise equality implies that the square commutes and it follows that $\nu$ is birational and projective as well.
    
    It remains to show that $\nu$ is an isomorphism. By the above and Proposition \ref{prop:mor-to-nqv}(\ref{prop:mor-to-nqv:4}), $\nu$ is projective, birational, and injective on closed points. Therefore, it must be surjective as well, because its image is closed and contains a dense open subset. It is then an isomorphism by Zariski's main theorem.
\end{proof}

\section{The nef and movable cones}\label{sec:cones}

We now describe the nef and movable cones of $\Hilb^n(\Sstack_K) \cong \Quot^{n\delta}_J$ and of $\Hilb^n(S_K)$. Since we have an isomorphism $\Quot^{n\delta}_J \cong \NQV_{\theta_K}(n\delta,\Lambda_0 )$ for $\theta_K \in C_K$, and $\NQV_{\theta_K}(n\delta,\Lambda_0 )$ is a fine moduli space, we identify the tautological bundles on both, and also denote the bundle on $\Quot^{n\delta}_J$ by $\mathcal{R}=\bigoplus_{i \in I} \mathcal{R}_i$ (this bundle $\mathcal{R}$ is the same as the one denoted $p_* \mathcal{F}$ in the proof of Proposition \ref{prop:mor-to-nqv}).

\begin{lemma}
    The movable cone of $\Quot^{n\delta}_J$ is equal to $\left\{ \bigotimes_{i \in I} \det(\mathcal{R}_i)^{\otimes \theta(\rho_i)} \mid \theta \in F \right\}$ and the ample cone is equal to $\left\{ \bigotimes_{i \in I} \det(\mathcal{R}_i)^{\otimes \theta(\rho_i)} \mid \theta \in C_K \right\}.$
\end{lemma}
\begin{proof}
    This is immediate from \cite[Proposition 6.1]{bellamy_craw_birational-symplectic_20}.
\end{proof}

The description of the Nef cone of $\Hilb^n(S_K)$ given in \cite{craw_wye_hilb-crepant-partial_25-v2} is somewhat unsatisfactory because it is phrased in terms of the tautological bundles on $\Quot^{n\delta}_J$. In light of the construction in section \ref{sec:geometric}, we can improve the statement of \cite[Theorem 1.3]{craw_wye_hilb-crepant-partial_25-v2}. 

\begin{theorem}\label{thm:NefHilb}
    For $J \coloneqq I \setminus K$, we have that:
    \begin{enumerate}
        \item[\one] the rational Picard group $\Pic(\Hilb^n(S_K))\otimes_\ZZ \QQ$ has basis $\{\det\big(\mathcal{T}_j^{[n]}\big) \mid j\in J\};$
        \item[\two] the nef cone is 
    \[
    \Nef\big(\Hilb^n(S_K)\big)=\Big\{\textstyle{\bigotimes_{j\in J} \det\big(\mathcal{T}_j^{[n]}\big)^{\otimes \eta(\rho_j)}} \mid \eta(\rho_j) \geq (n-1) \eta(\delta\vert_J)\geq 0 \text{ for }j\in J\setminus \{0\}\Big\};
    \]  
    \item[\three] the movable cone of $\Hilb^n(S_K)$ is the simplicial cone
   \[
    \Mov\big(\Hilb^n(S_K)\big)
     = 
     \Big\{\textstyle{\bigotimes_{j\in J} \det\big(\mathcal{T}_j^{[n]}\big)^{\otimes \eta(\rho_j)}} \mid 
     \eta(\rho_j) \geq 0 
     \text{ for }j\neq 0, \; 
     \eta(\delta\vert_J)\geq 0\Big\}
    \]
    with decomposition into Mori chambers determined by the hyperplanes $\{(m\delta_{J}- \alpha)^\perp \mid \alpha\in \Phi_{J \setminus \{0\}}^+, 0\leq m < n\}$, where $\Phi_{J \setminus \{0\}}^+$ is the set of positive roots of the root system of type ADE associated to $\Gamma$ that are supported in $J\setminus \{0\}$.
 \end{enumerate} 
\end{theorem}

The key observation is the relationship between the bundles $\mathcal{R}_j$ and $\mathcal{T}_j^{[n]}$. We have the following commutative diagram:
\begin{equation}
\label{eq:quothilbquiver}
   \begin{tikzcd} \Quot_J^{n \delta} \arrow[r, "\nu"] \arrow[d, "\xi"] & \NQV_{\theta_K}(n\delta,\Lambda_0 ) \arrow[d, "\Tilde{\lambda}"] \\ \Hilb^n(S_K) \arrow[r, "\nu_J"] & \NQV_{\theta'_K}(n\delta,\Lambda_0 ) \end{tikzcd}
\end{equation}
where the horizontal maps are isomorphisms, as in Theorem \ref{thm:isom-quot-nqv}.

\begin{lemma}
\label{lem:TjRj}
 For $j\in J$, the locally-free sheaves $\mathcal{T}_j^{[n]}$ on $\Hilb^n(S_K)$ and $\mathcal{R}_j$ on $\Quot^{n\delta}_J$ satisfy
 \[
 \xi^*(\mathcal{T}_j^{[n]})\cong \mathcal{R}_j.
 \]
\end{lemma}
 \begin{proof}
 We prove this statement by showing the two bundles on $\Quot^{n\delta}_J$ are pulled back from the same bundles on $\NQV_{\theta'_K}(n\delta,\Lambda_0 )$. For this, it is convenient to extend the diagram \eqref{eq:quothilbquiver} slightly to include the fine moduli space $\mathcal{M}_\eta(A_{\Bar{J}},(1, n\delta\vert_J))$ from section \ref{sec:geometric}:
\begin{equation}
\label{eq:quothilbquiverextended}
   \begin{tikzcd} \Quot_J^{n \delta} \arrow[r, "\nu"] \arrow[d, "\xi"] & \NQV_{\theta_K}(n\delta,\Lambda_0) \arrow[d, "\Tilde{\lambda}"] &\\ \Hilb^n(S_K) \arrow[r, "\nu_J"] & \NQV_{\theta'_K}(n\delta,\Lambda_0 ) \arrow[r, hook, "\varphi"]   & \mathcal{M}_\eta(A_{\Bar{J}},(1, n\delta\vert_J))
   \end{tikzcd}
\end{equation}
The fine moduli space $\mathcal{M}_\eta(A_{\Bar{J}},(1, n\delta\vert_J))$ carries a tautological bundle $\mathcal{U}= \bigoplus \mathcal{U}_j$. We show $\Tilde{\lambda}^*(\varphi^*(\mathcal{U}_j))\cong \mathcal{R}_j$ and $\mathcal{T}_j^{[n]}=\nu_J^*(\varphi^*(\mathcal{U}_j))$, thus showing $\xi^*(\mathcal{T}_j^{[n]}) \cong \nu^*\mathcal{R}_j$. Since we can identify the bundles on $\NQV_{\theta_K}(n\delta, \Lambda_0 )$ and on $\Quot^{n\delta}_J$, we will be done.

Each summand $\mathcal{U}_j$ is the descent to $\mathcal{M}_{\eta}(A_{\Bar{J}},(1, n\delta\vert_J))$ of the $G(1, n\delta\vert_J)$-equivariant locally-free sheaf $\varrho_j\otimes \mathcal{O}$ on the $\eta$-stable locus in $\Rep(A_{\Bar{J}},(1, n\delta\vert_J))$, where $\varrho_j\colon G(1, n\delta\vert_J)\to \GL(n\delta_j)$ is defined for each $j\in J$ by setting $\varrho_j(g)=g_j$ for any $g = (g_j)\in  G(1, n\delta\vert_J)$. The closed immersion $\varphi$ is defined in terms of a $G(1, n\delta\vert_J)$-equivariant closed immersion of representations schemes \cite{craw_yamagishi_hilb-can_26}, so $\varphi^*(\mathcal{U}_j)$ is the descent to $\NQV_{\theta'_K}(n\delta, \Lambda_0)$ of the $G(1, n\delta\vert_J)$-equivariant locally-free sheaf $\varrho_j\otimes \mathcal{O}$ on the $\theta'_K$-semistable locus in $\Rep(\Bar{\Pi},v)$. The morphism $\Tilde{\lambda}$ is a VGIT morphism, and so $\Tilde{\lambda}^*(\varphi^*(\mathcal{T}_j^{[n]}))$ is the descent of the $G(v)$-equivariant locally-free sheaf $\varrho_j\otimes \mathcal{O}$ restricted to the $\theta$-stable locus in $\Rep(\Bar{\Pi},v)$.  But this is the locally-free sheaf $\mathcal{R}_j$ on $\NQV_{\theta_K}(n\delta,\Lambda_0)$ by \cite[Proof of Lemma~3.11]{Craw23}, so $\Tilde{\lambda}^*(\varphi^*(\mathcal{U}_j))\cong \mathcal{R}_j$ for all $j\in J$.

The universal morphism $\nu' \colon \Hilb^n(S_K) \rightarrow \mathcal{M}_{\eta}(A_{\Bar{J}},(1, n\delta\vert_J))$ satisfies $\nu'^*(\mathcal{U}_j) = (\mathcal{F}_n)_j = \mathcal{T}_j^{[n]}$ by construction. Therefore, \[ \mathcal{T}_j^{[n]}=\nu'^*(\mathcal{U}_j)=(\varphi \circ \nu_J)^*(\mathcal{U}_j) = \nu_J^*(\varphi^*(\mathcal{U}_j)).\]
By commutativity of the diagram, $\nu^*(\Tilde{\lambda}^*(\varphi^*(\mathcal{U}_j))) = \xi^*(\nu_J^*(\varphi^*(\mathcal{U}_j)))$, so we can conclude $\nu^*\mathcal{R}_j \cong \xi^*(\mathcal{T}_j^{[n]})$ as required.
\end{proof}

Since $\nu$ and $\nu_J$ are isomorphisms, and $\xi$ is identified with $\Tilde{\lambda}$, we can equivalently say $\mathcal{R}_j\cong \Tilde{\lambda}^*\big(\mathcal{T}_j^{[n]}\big)$ as locally-free sheaves on $\NQV_{\theta_K}(n\delta,\Lambda_0 )$.

\begin{proof}[Proof of Theorem \ref{thm:NefHilb}]
    The proof now follows analogously to Theorem 1.3 of \cite{craw_wye_hilb-crepant-partial_25-v2}. For completeness, we provide an explanation. Part \one\ follows once we prove part \two, as the nef cone is top-dimensional in $\Pic(\Hilb^n(S_K))\otimes_\ZZ \QQ$. By the proof of Theorem 1.3 of \cite{craw_wye_hilb-crepant-partial_25-v2}, the map $\Tilde{\lambda}^*$ on rational Picard groups defines an equality
 \[
 \Tilde{\lambda}^*\big(\Nef(\NQV_{\theta'_K}(n\delta,\Lambda_0 ))\big) = \Big\{\bigotimes_{j\in J}\det(\mathcal{R}_j)^{\otimes \eta(\rho_j)} \mid \eta(\rho_j)\geq (n-1)\eta(\delta\vert_J)\geq 0 \text{ for }j\in J\setminus \{0\}\Big\}.
 \]
 Substitute $\mathcal{R}_j\cong \Tilde{\lambda}^*\big(\mathcal{T}_j^{[n]}\big)$ from Lemma~\ref{lem:TjRj} into the right-hand-side. Since $\Tilde{\lambda}^*$ is injective and pullback commutes with taking determinants and tensor products, we obtain 
 \[
 \Nef\!\big(\NQV_{\theta'_K}(n\delta,\Lambda_0 )\big) = \Big\{\bigotimes_{j\in J}\det\big(\mathcal{T}_j^{[n]}\big)^{\otimes \eta(\rho_j)} \mid \eta(\rho_j)\geq (n-1)\eta(\delta\vert_J)\geq 0 \text{ for }j\in J\setminus \{0\}\Big\}.
 \]
 For part \three, let $F'$ denote the minimal face of $F$ containing $\sigma_K$. Again, the argument from the proof of Theorem 1.3 from \cite{craw_wye_hilb-crepant-partial_25-v2} shows that the pullback of the movable cone of $\M_{\theta'_K}(n\delta,\Lambda_0 )$ is isomorphic to
\[
 L_{C_K}(F') = \Big\{\bigotimes_{i \in I}\det(\mathcal{R}_i)^{\otimes \vartheta_i} \mid \vartheta\in F'\Big\}.
 \] 
 The argument from part \two\ shows that we can replace the $\mathcal{R}_j$ with $\Tilde{\lambda}^*(\mathcal{T}_j^{[n]})$, so the movable cone is indeed the decomposition of  
   \[
    \Mov\big(\Hilb^n(S_K)\big)
     = 
     \Big\{\textstyle{\bigotimes_{j\in J} \det\big(\mathcal{T}_j^{[n]}\big)^{\otimes \eta(\rho_j)}} \mid 
     \eta(\rho_j) \geq 0 
     \text{ for }j\neq 0, \; 
     \eta(\delta\vert_J)\geq 0\Big\}
    \]
    determined by the hyperplanes in $\mathcal{A}$ that intersect $\relint(F')$; these are precisely   
    $(m\delta\vert_J- \alpha)^\perp$ for $\alpha\in \Phi_{J \setminus \{0\}}^+$ and $0\leq m < n$, where $\Phi_{J \setminus \{0\}}^+$ is the set of positive roots supported in $J\setminus \{0\}$.
\end{proof}

Theorem~\ref{thm:NefHilb} holds even for $n=1$, but in that case, the integral functor $p_*q^*(-)$ is the identity, so the line bundle $\det(\mathcal{T}_0^{[1]})\cong \mathcal{T}_0$ on $\Hilb^1(S_K)\cong S_K$ coincides with $\mathcal{O}_{S_K}$.

\bibliographystyle{amsplain}
\bibliography{References}

\end{document}